\documentclass[11pt]{amsart}
\usepackage{amsfonts, amssymb, latexsym, epsfig} 

\usepackage{amscd,amssymb, mathtools, etoolbox}
\usepackage{verbatim}
\usepackage[mathscr]{euscript}
\usepackage{amsmath}
\input xy
\xyoption{all}

\newsymbol\pp 1275

\usepackage{tikz, standalone, float}
\usetikzlibrary{shapes,arrows}

\usetikzlibrary{matrix,arrows,backgrounds,shapes.misc,shapes.geometric,patterns,calc,positioning}
\usetikzlibrary{patterns,arrows,chains,matrix,positioning,scopes,decorations.pathreplacing,decorations.pathmorphing,calc}

\usepackage[colorlinks=true, pdfstartview=FitV, linkcolor=blue, citecolor=blue, urlcolor=blue]{hyperref}

\usepackage{fullpage, subcaption}
\usepackage{graphicx}
\usepackage{enumerate}
\def\VR{\kern-\arraycolsep\strut\vrule &\kern-\arraycolsep}
\def\vr{\kern-\arraycolsep & \kern-\arraycolsep}

\usepackage{tikz-cd}
\usepackage{extarrows}

\pgfdeclarelayer{edgelayer}
\pgfdeclarelayer{nodelayer}
\pgfdeclarelayer{background}
\pgfdeclarelayer{foreground}
\pgfsetlayers{background,foreground,edgelayer,nodelayer,main}

\usepackage{blkarray}

\newcommand{\be}{\begin{enumerate}}

\newtheorem{theorem}{Theorem}

\newtheorem{prop}{Proposition}
\newtheorem{lemma}[prop]{Lemma}
\newtheorem{corollary}[prop]{Corollary}

\theoremstyle{definition}
\newtheorem{definition}[prop]{Definition}
\newtheorem{notation}[prop]{Notation}

\newtheorem{rmk}{Remark}

\newtheorem{obs}{Observation}

\newtheorem{ex}{Example}

\makeatletter
\tikzset{join/.code=\tikzset{after node path={%
\ifx\tikzchainprevious\pgfutil@empty\else(\tikzchainprevious)%
edge[every join]#1(\tikzchaincurrent)\fi}}}
\makeatother
\tikzset{>=stealth',every on chain/.append style={join},
         every join/.style={->}}
\tikzstyle{labeled}=[execute at begin node=$\scriptstyle,
   execute at end node=$]
   
   \makeatletter
\tikzset{join/.code=\tikzset{after node path={%
\ifx\tikzchainprevious\pgfutil@empty\else(\tikzchainprevious)%
edge[every join]#1(\tikzchaincurrent)\fi}}}
\makeatother

\begin{document}
\title{An Isometry Theorem for Generalized Persistence Modules}

\author{Killian Meehan, David Meyer}



\maketitle
\setcounter{tocdepth}{1}
\tikzset{
    >=stealth',
    punkt/.style={
           rectangle,
           rounded corners,
           draw=black, very thick,
           text width=6.5em,
           minimum height=2em,
           text centered},
    pil/.style={
           ->,
           thin,
           shorten <=2pt,
           shorten >=2pt,},
     bounce1/.style={
           thin,
           shorten <=2pt,},
     bounce2/.style={
           ->,
           thin,
           shorten >=2pt,},
     posetb/.style={
           <-,
           shorten >=2pt,
           shorten <=2pt,},
     posetf/.style={
           ->,
           shorten >=2pt,
           shorten <=2pt,}
}
\begin{abstract}
In recent work, generalized persistence modules have proved useful in distinguishing noise from the legitimate topological features of a data set.  Algebraically, generalized persistence modules can be viewed as representations for the poset algebra.  The interplay between various metrics on persistence modules has been of wide interest, most notably, the isometry theorem of Bauer and Lesnick for (one-dimensional) persistence modules.  The interleaving metric of Bubenik, de Silva and Scott endows the collection of representations of a poset with values in any category with the structure of a metric space.  This metric makes sense for any poset, and has the advantage that post-composition by any functor is a contraction.  In this paper, we prove an isometry theorem using this interleaving metric on a full subcategory of generalized persistence modules for a large class of posets.
\end{abstract}
\section{Introduction}
\subsection{Persistent Homology}
\noindent
Informally, a \emph{generalized persistence module} is a representation of a poset $P$ with values in a category $\mathcal{D}$.  That is, if $\mathcal{D}$ is a category, a generalized persistence module $M$ with values in $\mathcal{D}$ assigns an object $M(x)$ of $\mathcal{D}$ for each $x \in P$, and a morphism $M(x \leq y)$ in ${Mor}_{\mathcal{D}}(M(x),M(y))$ for each $x, y \in P$ with $x \leq y$ satisfying
$$M(x \leq z)=M(y \leq z) \circ M(x \leq y) \textrm{ whenever }x, y , z \in P \textrm{ with }x \leq y \leq z.$$  

Perhaps surprisingly, the study of such objects is useful in topological data analysis.  \emph{Persistent homology} uses  generalized persistence modules to attempt to discern the topological properties of a finite data set.  We briefly summarize the algorithm applied to a point cloud of data in the persistent homology setting.  This will lead to \emph{one-dimensional (generalized) persistence modules}, where the poset $P = (0,\infty)$ or $\mathbb{R}$.  This part of the discussion corresponds to $F$ and then $H \circ F$ in our schematic below.  For a more extensive introduction, see \cite{computing}, \cite{topological} or \cite{oudot}.  The typical workflow for persistent homology is as follows: 
\vspace{.2in} 
\tikzstyle{decision} = [diamond, draw, fill=blue!20, 
    text width=4.5em, text badly centered, node distance=3cm, inner sep=0pt]
\tikzstyle{block} = [rectangle, draw, fill=white, 
    text width=5em, text centered, rounded corners, minimum height=4em]
\tikzstyle{line} = [draw, -latex']
 \tikzstyle{circ} = [draw, ellipse,fill=red!20, node distance=3cm,
     minimum height=2em]

\tikzstyle{colored1} = [rectangle, draw, fill=blue!20, 
    text width=5em, text centered, rounded corners, minimum height=4em]
    
\begin{tikzpicture}[node distance =4cm, auto]
\node [cloud, draw,cloud puffs=10,cloud puff arc=120, aspect=2, inner ysep=1em] (data) {Data} ;
    \node [block, right of=data](comp) {Persistence modules with values in $Simp$};
    \node [colored1, right of=comp] (module) {Persistence modules with values in $K$-mod};
    \node [block, right of=module] (rank) {Topological space of invariants};
    \path [line](data) -- node {$F$} (comp);
    \path [line] (comp) -- node {$H$} (module);
    \path [line] (module) --  node {$J$} (rank);
\end{tikzpicture}

\vspace*{.2 in}
\noindent
(In the above, $Simp$ denotes the category of abstract simplicial complexes) Suppose, for example, we wish to decide whether a data set $D \subseteq {\mathbb{R}}^2$ should be more correctly interpreted as an annulus or a disk.  In order to decide between the two candidates, one calculates the homology of a filtration of simplicial complexes associated to the data set.  This uses the Vietoris-Rips complex ${(C_{\epsilon})}_{\epsilon > 0}$.

Specifically, for each $\epsilon > 0$, we let $C_{\epsilon}$ be the abstract simplicial complex whose $k$-simplices are determined by data points $x_1, x_2, ... x_{k+1}\in D$ where $d(x_i, x_j) \leq \epsilon$ for all $1 \leq i, j \leq k+1$.  Clearly, for ${\sigma} \leq {\tau}$ in $(0, \infty)$,  there is an inclusion of simplicial complexes $C_{\sigma} \hookrightarrow C_{\tau}$, thus we obtain a filtration of simplicial complexes indexed by $(0,\infty)$.  Therefore, the assignment $F:\epsilon \to C_{\epsilon}$ is a representation of the poset $(0, \infty)$ (in fact, a $P$-space) taking values in $Simp$.  That is to say, $F$ is a generalized persistence module for $P =(0,\infty)$ and $\mathcal{D} = Simp$.  Since we wish to distinguish between an annulus and a disk, we apply the first homology functor $H_1(-,K)$ to $F$ (where $K$ is some field), to obtain the representation of $P$ with values in $K$-mod, $\epsilon \to H_1(C_{\epsilon},K)$.

Thus, the assigment $H_1(-,K) \circ F \textrm{ given by } \epsilon \to H_1(C_{\epsilon},K)$ is a one-dimensional persistence module.  As $\epsilon$ increases generators for $H_1$ are born and die, as cycles appear and become boundaries.  In persistent homology, one takes the viewpoint that true topological features of the data set can be distinguished from noise by looking for generators of homology which "persist" for a long period of time.  Informally, one "keeps" an indecomposable summands of $H_1(-,K) \circ F$ when it corresponds to a wide interval.  Conversely, cycles which disappear quickly after their appearance (narrow ones) are interpreted as noise and disregarded.  

This technique has been widely successful in topological data analysis (see, for example, \cite{carlsson_top}, \cite{stability}, \cite{chazal}, \cite{ghrist}, \cite{carlsson_local}, \cite{applied_1}, \cite{applied_2}, and \cite{applied_3}).  Typically, the category of persistence modules with values in $K$-mod is given a metric-like structure.  So-called \emph{soft stability theorems}, which involve the continuity of the composition $H \circ F$ in the schematic, have been proven.  Philosophically, these results have established the utility of this method from the perspective of data analysis  (See, for example, \cite{stability}). \emph{Hard stability theorems}, on the other hand, concern the continuity of $J$ in our schematic.

\subsection{Algebraic Stability}
One special type of hard stability theorem is an \emph{algebraic stability theorem}.  In such a theorem, one endows a collection of generalized persistence modules with two metric structures, and  an automorphism $J$ is shown to be a contraction or an isometry.  This situation can be fit into our previous workflow diagram by choosing the topological space of invariants to be the collection of generalized persistence modules itself, endowed with the alternate metric structure.  Of particular interest is the case when $J$ is the identity function and the metrics are an \emph{interleaving metric} and a \emph{bottleneck metric}.  Algebraic stability theorems of this type are common (see \cite{lesnick}, \cite{zigzag}, \cite{induced_matchings}, and \cite{carlsson}).  While in the literature, the word "interleaving" is frequently used to describe slightly different metrics, we believe that the interleaving metric suggested by Bubenik, de Silva and Scott (see \cite{bubenik}) has the advantage of being both most general, and categorical in nature.  This interleaving metric makes sense on any poset $P$, and reduces to the interleaving metric of \cite{induced_matchings} when $P = (0,\infty)$.  Alternatively, a bottleneck metric is nothing more than a way of extending a metric defined on a set ${\Sigma}$, to the collection of all ${\mathbb{Z}}_{\geq 0}$-valued functions with finite support on ${\Sigma}$.  In this context, this is applied to the decomposition of a generalized persistence module into its indecomposable summands with their corresponding multiplicities.

\subsection{Connections to Finite-dimensional Algebras}
This paper concerns algebraic stability studied using techniques from the representation theory of algebras.  Such representations appear because one-dimensional persistence modules arising from data always admit the structure of a representation of a finite totally ordered set.  This fact comes from the simple observation that the one-dimensional persistence module given by $F:\epsilon \to C_{\epsilon}$ is necessarily a step function.  More precisely, let 
$$P_n = \{\epsilon_1 < \epsilon_2 <... < \epsilon_n\}  = \{ \epsilon \in (0,\infty): C_{\epsilon} \neq \lim \limits_{\tau \to {{\epsilon}^-}} C_{\tau} \}.$$

By definition, $F$ is constant on all intervals of the form $[\epsilon_i,\epsilon_{i+1})$.  Thus, clearly, both $F$ and $H_1(-,K) \circ F$ admit the structure of a generalized persistence module for $P= P_n$.  
When we restrict the structure of a one-dimensional persistence module to $P_n$, we say informally that we are \emph{discretizing}.  In this sense, generalized persistence modules for finite totally ordered sets are the discrete analogue of one-dimensional persistence modules.  At this point the authors wish to point out two issues arising when one discretizes.  First, a finite data set $D$ gives rise to not only a generalized persistence module, but also to its algebra.  Thus, a priori two persistence modules may not be able to be compared simply because they are not modules for the same algebra.  Second, information about the width of the interval $[\epsilon_i, \epsilon_{i+1})$ in relevant to the analysis, but seems to be lost.  Both of these issues are not addressed in this paper, though they are dealt with successfully in \cite{meehan_meyer_2}.

While one-dimensional persistence modules will always discretize to a generalized persistence module for a finite totally ordered set, representations of many other infinite families of finite posets also have a physical interpretation in the literature (see \cite{zigzag}, \cite{carlsson}, \cite{ladder}).  For example, multi-dimensional persistence modules (see \cite{carlsson}) will discretize in an analagous fashion to representations of a different family of finite posets.  This is relevant because there is a categorical equivalence between the generalized persistence modules for a finite poset $P$ with values in $K$-mod, and the module category of the finite-dimensional $K$-algebra $A(P)$, the poset (or incidence) algebra of $P$.  The module theory (representation theory) of such algebras has been widely studied (see, for example \cite{strongly_simply_connected}, \cite{baclawski}, \cite{cibils}, \cite{feinberg}, \cite{kleiner}, \cite{tame}, \cite{loupias}, \cite{nazarova}, \cite{quadratic}, \cite{mio}, and many others).  Thus, by passing to the jump discontinuities of a filtration of simplical complexes one may apply techniques from the representation theory of finite-dimensional algebras.

This perspective, however, suggests the need for caution.  While it is well-known that the set of isomorphism classes of indecomposable modules for the algebra $A(P_n)$ is \emph{finite}, this situation is far from typical.  In fact, for a \emph{generic} finite poset $P$, the representation theory of the algebra $A(P)$ is undecidable in the sense of first order logic.

In the above, generic means for all but those on a known list.  In particular, for all $P$ not on the list, the algebra $A(P)$ has infinitely many isomorphism classes of indecomposable modules.  Indeed, this is the case for the algebras associated to many of the posets which arise when one discretizes in a situation pertinent to topological data analysis.  This typically happens for multi-dimensional persistence modules (see \cite{carlsson}), for example.  

Because of this, studying \emph{arbitrary} generalized persistence modules in complete generality is hopeless.  Indeed, if a possibly infinite poset discretizes to a finite poset $P$, and the module category for $A(P)$ is undecidable, the same holds for generalized persistence modules for the original poset.  Moreover, our intuition from persistent homology tells us that indecomposable modules should come with a notion of widths which can be measured, in order to decide whether they should be kept or interpreted as noise.  In order to reconcile these two issues, we pass from the full category of all $A(P)$-modules, to a more manageable full subcategory where we can make sense of what it means for indecomposable modules to be "wide."  This suggests the following template for a representation-theoretic algebraic stability theorem:

\vspace{.1 in}
\noindent
Let $P$ be a finite poset of some prescribed type, and let $K$ be a field.  Choose a full subcategory  $\mathcal{C} \subseteq A(P)$-mod, and let $D$ and $D_B$ be two metrics on $\mathcal{C}$ where;
\begin{enumerate}[(i.)]
\item $D$ is the interleaving distance of \cite{bubenik} restricted to $\mathcal{C}$, and
\item $D_B$ is a bottleneck metric on $\mathcal{C}$ which incorporates some algebraic information.
\end{enumerate}
Prove that
$$(\mathcal{C},D) \xrightarrow{Id} (\mathcal{C},D_B) $$
is an isometry.

\medskip
In addition, the class of posets covered should contain all the posets $P_n, n \in \mathbb{N}$.  In addition, the category $\mathcal{C}$ should reduce to the full module category when $P = P_n$.  When this is the case, the theorem should be a discrete version of the classical isometry theorem \cite{induced_matchings}.  If possible, elements of $\mathcal{C}$ should have a nice physical description.
\medskip

\subsection{Main Results}
Our algebraic stability theorem is stated below.  
\vspace{.3 in}

\begin{theorem}
\label{main}
Let $P$ be an $n$-Vee and let $\mathcal{C}$ be the full subcategory of $A(P)$-modules consisting of direct sums of convex modules.  Let $(a,b) \in {\mathbb{N}} \times {\mathbb{N}}$ be a weight and let $D$ denote interleaving distance (corresponding to the weight $(a,b)$) restricted to $\mathcal{C}$.  \\Let  $W(M) = \textrm{min}\{ \epsilon: \textrm{Hom}(M, M\Gamma \Lambda) = 0, \Gamma, \Lambda \in  \mathcal{T}(\mathcal{P}), h(\Gamma), h(\Lambda) \leq \epsilon \}$, and let $D_B$ be the bottleneck distance on $\mathcal{C}$ corresponding to the interleaving distance and $W$.  Then, the identity is an isometry from $$(\mathcal{C},D) \xrightarrow{Id} (\mathcal{C},D_B).$$
\end{theorem}

Of course, much of the language in the Theorem has not yet been defined.  The collection of $n$-Vees generalizes $\{P_m\}$ in the sense that a $1$-Vee is exactly a finite totally ordered set.  Such a theorem is very much in the flavor of classical algebraic stability theorems (see \cite{zigzag},\cite{lesnick}, \cite{bubenik}).  It is common, for example, for the bottleneck metric restricted to indecomposables to be the  interleaving metric.  When $P$ is a $1$-Vee and the choice of weight is $(1,1)$, Theorem \ref{main} is a discrete analogue of the standard isometry theorem of Bauer and Lesnick \cite{induced_matchings}, though with a different notion of width, and with the interleaving metric of \cite{bubenik}.  Indeed, both Theorems \ref{matching} and \ref{main} can be viewed as extensions of the discrete analogue to the classical isometry theorem \cite{induced_matchings}.

In the statement Theorem \ref{main}, $W$ corresponds to our choice for the width function.  We take our inspiration for $W$ from \cite{zigzag} and \cite{lesnick}, but do not use the thickness of the support of a module $M$ in a direction.  Instead, our width is defined in terms of algebraic conditions, although the two agree in the case of one-dimensional persistence modules.  Our choice of the category $\mathcal{C}$ is natural both from the perspective of persistent homology and from that of representation theory.  Once some parameters are fixed, the collection of interleavings between two elements of $\mathcal{C}$ has the structure of an affine variety (see Proposition \ref{variety} and Examples \ref{ex 3}, \ref{ex 4}, and \ref{ex new}).  The interleaving distance between two generalized persistence modules is the smallest value of a parameter for which the corresponding variety of interleavings in non-empty (see Remark \ref{variety distance} and  Example \ref{ex new}).

While certainly motivated by stability theorems in topological data analysis, the authors take the viewpoint that such a theorem need not make explicit reference to a data set.  This paper will be organized as follows:  in Section \ref{preliminaries} we give a brief survey of the relevant background information, in Section \ref{posets} we define the class of posets in which we will work, and in Section \ref{homomorphisms and translations} we investigate the action of the collection of translations on the set of homomorphisms between convex modules.  Then, in Sections \ref{totally} we concentrate on $1$-Vees, that is, totally ordered finite sets.  In Section \ref{totally} in particular, we owe much to Bauer and Lesnick \cite{induced_matchings}.  Then, in Section \ref{main section} we prove our main results.  After the proof of the main results we include some examples.

\section{Acknowledgements}
The authors wish to acknowledge Calin Chindris both for introducing us to this field of study, and for all of his guidance.  K. Meehan was supported by the NSA under grant H98230-15-1-0022.

\section{Preliminaries}
\label{preliminaries}
\subsection{Generalized Persistence Modules}
Recall that if $P$ is a poset and $\mathcal{D}$ is a category, a generalized persistence module $M$ with values in $\mathcal{D}$ assigns an object $M(x)$ of $\mathcal{D}$ for each $x \in P$, and a morphism $M(x \leq y)$ in ${Mor}_{\mathcal{D}}(M(x),M(y))$ for each $x, y \in P$ with $x \leq y$ satisfying
$$M(x \leq z)=M(y \leq z) \circ M(x \leq y)\textrm{ whenever }x, y , z \in P \textrm{ and }x \leq y \leq z.$$ 

Let ${\mathcal{D}}^P$ denote the collection of generalized persistence modules for $P$ with values in $\mathcal{D}$.  If $F, G \in {\mathcal{D}}^P$, a morphism from $F$ to $G$ is a collection of morphisms $\{\phi(x)\}$, with $\phi(x) \in {Mor}_{\mathcal{D}}(F(x),G(x))$ for all $x \in P$, such that for all $x \leq y$ we have a commutative diagram below for each $x \leq y$ in $P$.

\begin{center}

\begin{tikzpicture}[commutative diagrams/every diagram]
	\matrix[matrix of math nodes, name=m, commutative diagrams/every cell,row sep=.7cm,column sep=1cm] {
	 \pgftransformscale{0.2}
		F(x) & F(y)\\
		G(x) & G(y)\\ };
		
		\path[commutative diagrams/.cd, every arrow, every label]
			(m-1-1) edge node {$F(x \leq y)$} (m-1-2)
			(m-1-1) edge node[swap] {$\phi(x)$} (m-2-1)
			(m-2-1) edge node {$G(x \leq y)$} (m-2-2)
			(m-1-2) edge node {$\phi(y)$} (m-2-2);
		
\end{tikzpicture}

\end{center}
With these morphisms, ${\mathcal{D}}^P$  is a category.  Equivalently, one could regard the poset $P$ as a thin category.  Then, a generalized persistence module will correspond to a covariant functor from $P$ to $\mathcal{D}$, and morphisms in ${\mathcal{D}}^P$ will be natural transformations.  When $P$ is $(0,\infty)$ or $\mathbb{R}$, we say that the elements of ${\mathcal{D}}^P$ are one-dimesional persistence modules.  In this paper, $\mathcal{D}$ will always be  $Simp$ or $K$-mod.

\subsection{Representation Theory of Algebras}
\label{algebra}
In this subsection we give a brief summary of $K$-algebras and their representations (modules).  For a more expansive introduction, see \cite{auslander}, \cite{benson1}, \cite{benson2}.  Throughout, let $K$ denote a field.  If $R$ is a $K$-algebra, by an $R$-module, we mean a finite-dimensional, unital, left $R$-module.  The category $R$-mod consists of $R$-modules together with $R$-module homomorphisms.

Recall that an $R$-module $M$ is \emph{indecomposable} if it is not isomorphic to a direct sum of two of its proper submodules.  The category $R$-mod is an abelian Krull-Schmidt category.  That is, every module can be written as a direct sum of indecomposable modules in a unique way up to order and isomorphism.  Moreover, the decomposition of modules is compatible with respect to homomorphisms in the following sense.

\begin{prop}
Let $R$ be a $K$-algebra, and let $M, N$ be $R$-modules.  Say, $M \cong \oplus M_i$ and $N \cong \oplus N_j$.  Then, as vector spaces,
$$\textrm{Hom}(M,N) \cong \bigoplus\limits_{i,j} \textrm{Hom}(M_i,N_j). $$
\end{prop} 

This says that any module homomorphism can be $f:M \to N$ can be factored into a matrix of module homorphisms $f^i_j : M_i \to N_j$.

\subsubsection{Bound Quivers and their Representations}

\begin{definition}
A quiver $Q = (Q_0,Q_1, t, h)$ is an ordered tuple, where $Q_0, Q_1$ are disjoint sets, and $t, h : Q_1 \to Q_0$.
\end{definition}

We call elements of $Q_0$ vertices, and elements of $Q_1$ arrows.  The functions $t$ and $h$ denote the tail (start) and head (end) of the arrows.  Thus, clearly $Q$ is exactly a directed set.  We will always suppose the sets $Q_0, Q_1$ are finite.  

\begin{ex} 
\label{quivers}
Below are two quivers.

\begin{figure}[H]
\centering
\begin{subfigure}[b]{0.5\textwidth}
\label{quiver1}
\begin{tikzpicture}[normal line/.style={->},shorten >=1pt,font={\it\small},node distance=2cm,main node/.style={circle,scale=.5,fill=blue!10,draw,font=\sffamily\small\bfseries}]
 
  \node (1) {$\bullet_1$};
  \node (2) [right of=1] {$\bullet_2$};
  \node (3) [right of=2] {$\bullet_3$};
  \node (4) [above right of=3] {$\bullet_4$};
  \node (5) [below right of=3] {$\bullet_5$};

  \path[normal line]
    (1) edge node [above] {a} (2) 
    (2) edge node [above] {b} (3)
    (3) edge node [above] {c} (4)
    (3) edge node [above] {d} (5);

\end{tikzpicture}
\caption*{A}

\end{subfigure}
~
\begin{subfigure}[b]{0.5\textwidth}
\label{quiver2}
\centering
\begin{tikzpicture}[normal line/.style={->},shorten >=1pt,font={\it\small},node distance=3cm,main node/.style={circle,scale=.5,fill=blue!10,draw,font=\sffamily\small\bfseries}]
 
  \node (1) {$\bullet_1$};
  \node (2) [above right of=1] {$\bullet_2$};
  \node (3) [right of=1] {$\bullet_3$};
  \node (4) [right of=3] {$\bullet_4$};

  \path[normal line]
  	(1) edge node [right] {a} (2) 
	(2) edge [bend right=30] node [right] {c} (3)
	(3) edge node [below right] {b} (1)
	(3) edge [bend right=45] node [right] {d} (2)
	(3) edge [bend right=90] node [below] {e} (4)
	(3) edge [bend right] node [above] {f} (4)
	(4) edge [loop] node [above] {g} (4);

\end{tikzpicture}
\caption*{B}
\label{quiver2}

\end{subfigure}
\end{figure}
\noindent
Quiver A corresponds to $Q_0 = \{1,2,3,4,5\}, Q_1 = \{a,b,c,d\}$, for an appropriate choice of the functions $h, t$.  Similarly, quiver B corresponds to the sets $Q_0 = \{1, 2, 3, 4\}$, and $Q_1 = \{a, b, c, d, e, f, g\}$. 
\end{ex}
\begin{definition}
A path is a sequence of arrows $p = a_1 ... a_n$ where $t(a_i) = h(a_{i+1})$.  The length of the path is the number of terms in the sequence $p$.  In addition, at each vertex $i$ there is a "lazy" path $e_i$ of length $0$ at the vertex $i$.  We extend the functions $h, t$ to paths, by defining $t(p) = t(a_n)$ and $h(p) = h(a_1)$.   In addition, $t(e_i) = h(e_i) = i$.  An oriented cycle is a path $p$ of length greater than or equal to one with $t(p) = h(p)$.  
\end{definition}
Consider quiver B in Example \ref{quivers}.  Then $g$ and $cabcd$ are oriented cycles, while $gg f$ is a path which is not an oriented cycle.  Quiver A has no oriented cycles.

\begin{definition}
A representation $V$ of a quiver $Q$ is a family $V = (\{V(i)\}_{i \in Q_0}, \{V(a)\}_{a \in Q_1})$, where $V(i)$ is a $K$-vector space for every $i \in Q_0$, and $V(a) : V(t(a)) \to V(h(a))$ is a $K$- linear map for every $a \in Q_1$.
\end{definition}

For a fixed quiver $Q$ and field $K$, the collection of all representations of $Q$ is a category with morphisms given below.
\begin{definition}
Let $Q$ be a quiver, and let $V$, $W$ be representations of $Q$.  A morphism from $V$ to $W$, $\phi: V \to W$ is a collection of linear maps $\{\phi(i)\}_{i \in Q_0}$ with $\phi(i):V(i) \to W(i)$ such that the diagam below commutes for all $a \in Q_1$

\begin{center}

\begin{tikzpicture}[commutative diagrams/every diagram]
	\matrix[matrix of math nodes, name=m, commutative diagrams/every cell,row sep=.7cm,column sep=1cm] {
	 \pgftransformscale{0.2}
		V(t(a)) & V(h(a))\\
		W(t(a)) & W(h(a))\\ };
		
		\path[commutative diagrams/.cd, every arrow, every label]
			(m-1-1) edge node {$V(a)$} (m-1-2)
			(m-1-1) edge node[swap] {$\phi(t(a))$} (m-2-1)
			(m-2-1) edge node[swap] {$V(a)$} (m-2-2)
			(m-1-2) edge node {$\phi(h(a))$} (m-2-2);
		
\end{tikzpicture}

\end{center}

\end{definition}

\medskip
We denote by $Rep(Q)$, the category of $K$-representations of the quiver $Q$.  When $\phi: V \to W$ is a morphsim from $V$ to $W$ and $\phi(i)$ is invertible for all $i$, then we say $\phi$ is an isomorphism.  If this is the case, we say that $V$ and $W$ are isomorphic. 
\begin{definition}
\label{support}
If $V$ is a representation of a quiver $Q$ we say the support of $V$ is the set of all vertices $i \in Q_0$, such that $V(i)$ is not the zero vector space.
\end{definition}

More generally, the \emph{dimension vector} of $V$ is the non-negative integer vector $(dim_K(V(i)))$.  Viewing the dimension vector of $V$ as a function from $Q_0$ to the non-negative integers, the support of $V$ is exactly the support of this function.

\begin{definition}
Let $Q$ be a quiver.  The path algebra $KQ$ is the $K$-vector space with basis consisting of all paths (including those of length zero).  We define multiplication in $KQ$ as the $K$-linear extension of concatenation of paths.  
\end{definition}

That is, if $p, q$ are paths, then $p \cdot q = pq$, if $pq$ is a path, and zero otherwise.  If $t(p) = a, h(p) = b$, we define $p e_a = p = e_b p$.  By extending $K$-linearly, we obtain a ring structure on $KQ$.  It is easy to see that $KQ$ is finite-dimensional if and only if $Q$ has no oriented cycles.  

The (two-sided) ideal $J$ in $KQ$ generated by the arrows is the radical of the  ring $KQ$.  For $n \in N$, let $J^n$ denote the $n$th power of the radical $J$.  When $Q$ has no oriented cycles, $J$ is a nilpotent ideal.  We say an ideal $I$ is admissible if $J^n \subseteq I \subseteq J^2$, for some $n$.  The elements of $I$ are called \emph{relations}.  If $Q$ is a quiver, and $I$ is an admissible ideal, we say $(Q,I)$ is a \emph{bound quiver}.

\begin{definition}
Let $(Q,I)$ be a bound quiver.  Then, $Rep(Q,I)$ denotes the collection of all representations $V$ in $Rep(Q)$ satisfying all the relations in $I$. 
\end{definition}
Then, $Rep(Q,I)$ with morphisms in $Rep(Q)$ forms a category.

\begin{prop}
\label{quiver equivalence}
\label{bound quiver equivalence}
Let $(Q,I)$ be a bound quiver. Then, there exists a natural equivalence between $Rep(Q)$ and $KQ$-mod, that restricts to $Rep(Q,I)$ and $KQ/I$.
\end{prop}

Gabriel proved that when $K$ is algebraically closed any finite-dimensional $K$-algebra is (Morita equivalent to) an algebra of the form $KQ/I$.  Thus,  up to equivalence, the study of the module category of $K$-algebras (when $K$ is algebraically closed) \emph{is} the study of representations of bound quivers.

\subsubsection{Poset Algebras}

We will now define the algebra whose module theory is equivalent to generalized persistence modules with values in $K$-mod.  This will be $A(P)$, the poset algebra (or incidence algebra) of the poset $P$.

\begin{definition}
Let $P$ be a finite poset.  Let $Q_P$ be the quiver with $Q_0 = P$.  There is an $a \in Q_1$ with $t(a) = x, h(a) = y$ if,
\begin{enumerate}[(i.)]
\item $x < y$, and 
\item there is no $t \in P$, with $x < t < y$.
\end{enumerate}
\end{definition}

The quiver $Q_P$ is called the Hasse quiver of the poset $P$.  The Hasse quiver of $P$ is exactly the lattice of the poset with arrows corresponding to minimal proper relations.

\begin{ex}
\label{Hasse}

The quivers below are the Hasse quivers for three finite posets.
\begin{center}
\includegraphics[scale=2]{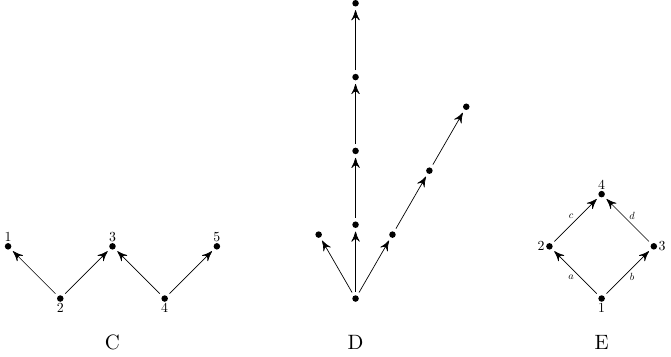}
\end{center}
Quiver A in Example \ref{quivers} is also the Hasse quiver of a poset.

\end{ex}

Note that if $Q_P$ is the Hasse quiver for $P$, and there is an arrow going from one vertex to another, it is necessarily unique.  Because no ambiguity is possible, we may draw the Hasse quiver of a poset with arrows unlabeled.  A finite quiver $Q$ is the Hasse quiver of a poset if and only if has no oriented cycles, and for all arrows $p$, if $q$ is a path of length greater than or equal to one with $t(p) = t(q)$ and $h(p) = h(q)$, then $p = q$.  Note that $x \leq y$ in $P$ if and only if there is a path $q$ in the Hasse quiver with $t(q) = x$ and $h(q) = y$.

\begin{definition}
Let $KQ_P$ denote the Hasse quiver of the poset $P$.  Then, the parallel ideal $I_P$ is the two-sided ideal in $KQ_P$ generated by all the relations equating any two paths $p, q$ in $Q_P$ satisfying $t(q) = t(p)$ and $h(q) = h(p)$.  
\end{definition}

For example, the poset in Example \ref{Hasse} E has parallel ideal generated by the element $ca-db$.  The Hasse quiver for C and D have trivial parallel ideals.

\begin{definition}
The poset algebra $A(P)$ is the bound quiver algebra
$$A(P) = KQ_P/I_P .$$
\end{definition}

By the equivalence in Proposition \ref{bound quiver equivalence}, we now see that the generalized persistence modules for a finite set $P$ with values in $K$-mod are the same as the modules for the poset algebra $A(P)$.  This is because $Rep(Q_P,I_P)$ corresponds exactly to the definition of generalized persistence modules for $\mathcal{D}=K$-mod, where the commutativity of the triangle below corresponds to the statment that $M$ satisfies all relations in $I_P$.

\begin{center}
\begin{tikzpicture}[commutative diagrams/every diagram]
	\matrix[matrix of math nodes, name=m, commutative diagrams/every cell,row sep=.7cm,column sep=.45cm] {
	 \pgftransformscale{0.2}
		M(x) & & M(z) \\
		& M(y) & \\ };
		
		\path[commutative diagrams/.cd, every arrow, every label]
			(m-1-1) edge node {$M(x \leq z)$} (m-1-3)
			(m-1-1) edge node[swap] {$M(x \leq y)$} (m-2-2)
			(m-2-2) edge node[swap] {$M(y \leq z)$} (m-1-3);

\end{tikzpicture}
\end{center}

Thus, from this point forward, we pass freely between generalized persistence modules and modules for the corresponding poset algebra.
\subsubsection{Representation Type}

For this subsubsection only, let $K$ be algebraically closed.  Informally, an arbitrary $K$-algebra $A$ is said to be of wild representation type if it's module category contains a copy of the module category of \emph{all} finite dimensional $K$-algebras.  Rather surprisingly, this happens frequently.

\begin{ex}
\label{3-Vee}
Poset D in Example \ref{Hasse} is a poset whose algebra is of wild representation type.
\end{ex}

When $A$ is of wild representation type, the classification of its modules up to isomorphism is hopeless.  In contrast, the module category for $A$ may be of finite type, or of tame type.  Finite representation type means that there are a finite number of isomorphism classes of indecomposable modules (like $A(P_n)$).  Informally, if $A$ has tame representation type there are infinitely many isomorphism classes of indecomposable $A$-modules (though they are parametrized reasonably).  It has been shown that every algebra is either finite, tame or wild.  In particular, complete lists of poset algebras of finite representation type are known (see \cite{loupias}, \cite{drozdowski}).  The  posets that arise when one discretizes generalized persistence modules for $P ={\mathbb{R}}^n$ are typically wild.

\subsection{Interleaving Metrics on $P$ and $P^+$}

We begin with the construction of the interleaving metric of Bubenik, de Silva and Scott (see \cite{bubenik}).
\begin{definition}
Let $P$ we a finite poset and $\mathcal{T}(P^-)$ be the collection of endomorphisms of the poset $P$ with the additional property that $\Lambda p \geq p$ for all $p \in P$.  We call the elements of $\mathcal{T}(P^-)$  translations.  
\end{definition}

Explicitly, a function $\Lambda: P \to P$ is an element of $\mathcal{T}(P^-)$ if and only if
$$x \leq y \implies \Lambda x \leq \Lambda y \textrm{, and }
\Lambda p \geq p\textrm{, for all } p \in P.$$

It is easy to see that the set $\mathcal{T}(P^-)$ is itself a poset under the relation $\Lambda \leq \Gamma$ if for all $p \in P$, $\Lambda p \leq \Gamma p$. Moreover $\mathcal{T}(P^-)$ is totally ordered if and only if $P$ is totally ordered, and $\mathcal{T}(P^-)$ is a monoid under functional composition.  Let $d$ be any metric on a finite poset $P$, we define a height function $h = h(d)$ on $\mathcal{T}(P^-)$.

\begin{definition}
For $\Lambda \in \mathcal{T}(P^-)$ set $h(\Lambda) = \textrm{sup}\{d(x,\Lambda x): x \in P \}$
\end{definition}

Of course, since $P$ is finite, we may replace supremum with maximum.  Proceeding as in \cite{bubenik}, let $\mathcal{D}$ be any category.  Then, $\mathcal{T}(P^-)$ acts on $\mathcal{D}^P$ on the right by the formulae
$$(F \cdot \Gamma)(p) = F(\Gamma p), \textrm{ and } (F \cdot \Gamma)(p \leq q) =F(\Gamma p \leq \Gamma q)\textrm{, for }\Gamma \in \mathcal{T}(P^-), F \in \mathcal{D}^P.$$
Similarly, $\mathcal{T}(P^-)$ acts on morphisms in $\mathcal{D}^P$, by acting inside the argument.
\begin{definition}
Let $F,G\in\mathcal{D}^P$ and let $\Gamma, \Lambda$ be translations on $P$. A $(\Gamma,\Lambda)$-interleaving between $F$ and $G$ is a pair of morphisms in $\mathcal{D}^P$, $\phi:F\to G\Lambda,\,\,\,\,\psi:G\to F\Gamma$ such that the following diagrams commute:

\begin{center}

\begin{tikzpicture}[commutative diagrams/every diagram]
	\matrix[matrix of math nodes, name=m, commutative diagrams/every cell,row sep=.7cm,column sep=.45cm] {
	 \pgftransformscale{0.2}
		F & & F\Gamma\Lambda  & & F\Gamma & \\
		& G\Lambda  & & G & & G\Lambda \Gamma\\ };
		
		\path[commutative diagrams/.cd, every arrow, every label]
			(m-1-1) edge node {} (m-1-3)
			(m-1-1) edge node[swap] {$\phi$} (m-2-2)
			(m-2-2) edge node[swap] {$\psi \Lambda $} (m-1-3)
			(m-2-4) edge node {$\psi$} (m-1-5)
			(m-2-4) edge node[swap] {} (m-2-6)
			(m-1-5) edge node {$\phi \Gamma$} (m-2-6);
					
\end{tikzpicture}

\end{center}

\end{definition}
The two horizontal maps in the diagram above are given by the formulae:
$$\textrm{for all }p \in P, F( p \leq \Gamma\Lambda p) \textrm{, and } G(p \leq \Lambda \Gamma p) \textrm{ respectively.}  $$
Note that two persistence modules are $(1,1)$-interleaved, where $1$ is identity translation, if and only if they are isomorphic.  
\begin{definition}[\cite{bubenik}]
Given any metric $d$ on $P$, we define $D = D(d)$ by the formula;
\begin{eqnarray*}
&&D(M,N) :=\textrm{inf} \{ \epsilon : \exists (\Gamma,\Lambda) \textrm{-interleaving with } {\textrm{sup}}_{p \in P}d(p,\Gamma p), {\textrm{sup}}_{p \in P}d(p,\Lambda p) \leq \epsilon \}\\
&&=\textrm{inf}\{\epsilon: \exists (\Gamma,\Lambda) \textrm{-interleaving with } h(\Lambda), h(\Gamma) \leq \epsilon\}.
\end{eqnarray*}
\end{definition}

From Bubenik, de Silva and Scott (\cite{bubenik}), we know that $D$ is a Lawvere metric on ${\mathcal{D}}^P$, and for any category $\mathcal{F}$, and functor $R: \mathcal{D} \to \mathcal{F}$, post-composition by $R$ is a contraction from ${\mathcal{D}}^P$ to ${\mathcal{F}}^P$.

With hard stability theorems in mind, the fact that post-composition by any functor induces a contraction is particularly noteworthy.  Still, independent of the choice of metric $d$ on $P$, without modification the resulting Lawvere metric $D=D(d)$ need not be a proper metric, simply because the collection of translations is not be rich enough to provide interleavings between arbitrary generalized persistence modules.  This is unfortunate, since $\mathcal{T}(P^-)$ is defined naturally for any poset $P$.  The failure comes from the fact that finite posets will always have fixed points.  
\begin{definition}
We say that $p \in P$ is a fixed point of $P$, if $\Lambda p = p$ for all $\Lambda$ in $\mathcal{T}(P^-)$  
\end{definition}
\begin{rmk}
\label{fixed}
Note that if $p_1, p_2 ... p_n$ are maximal elements in $P$, then any maximal element in $\bigcap (-\infty, p_i ] $ is necessarily a fixed point of $P$.  This is relevant because one can easily show that if $M$, $N$ are two $A(P)$ modules, and $dim_K (M(p)) \neq dim_K (N(p)) $ for some fixed point $p \in P$, then $D(M,N) = \infty$, where $D = D(d)$, and $d$ is any metric on $P$.
\end{rmk}

If $p$ is a fixed point of $P$ and $dim_K (M(p)) < dim_K (N(p)) $.  Then the diagram below does not commute for any morphisms $\phi, \psi$ and any translations $\Lambda, \Gamma$, since the composition cannot have full rank as required.  Thus, $D(M,N) = \infty$.

\begin{center}
\begin{tikzpicture}[commutative diagrams/every diagram]
	\matrix[matrix of math nodes, name=m, commutative diagrams/every cell,row sep=.7cm,column sep=.45cm] {
	 \pgftransformscale{0.2}
		N(p) & & N(\Gamma\Lambda p)=N(p) \\
		& M(\Lambda p)= M(p) & \\ };
		
		\path[commutative diagrams/.cd, every arrow, every label]
			(m-1-1) edge node {$N(p\leq \Gamma\Lambda p)= Id_{N(p)}$} (m-1-3)
			(m-1-1) edge node[swap] {$\phi_p$} (m-2-2)
			(m-2-2) edge node[swap] {$\psi_{\Lambda p}=\psi_p$} (m-1-3);

\end{tikzpicture}
\end{center}

In particular, this says that if $p$ is a fixed point of $P$, with $p \in \textrm{Supp}(M), p \notin \textrm{Supp}(N)$, then $D(M,N) = \infty$.  Because of this, there is no hope of realizing any honest metric as an interleaving metric on any finite poset.  For example, for poset C of Example \ref{Hasse}, and for any choice of metric $d$, the resulting interleaving metric (on isomorphism classes of modules) is the infinite discrete Lawvere metric.

With this in mind, we make the following modification.  We set $P^+ = P \cup \{\infty \}$ with added relations $p \leq \infty$, for all $p \in P$.  We may now view $A(P)$-mod as the full subcategory of $A(P^+)$-modules where all objects are supported in $P$.  Now there exist ($P^+$) interleavings between any two $A(P)$-modules.  Note that the Hasse quiver for $P^+$ is simply the Hasse quiver for $P$ with added edges connecting maximal elements of $P$ to $\infty$.  We now build the metric $d$, attaching positive weights to each edge of the Hasse quiver of $P^+$.  Continuing with poset C from Example \ref{Hasse}, we now have one of the below.


\begin{figure}[H]
\centering
\begin{subfigure}[b]{0.5\textwidth}
\label{quiver1}
\centering
\includegraphics[scale=2.25]{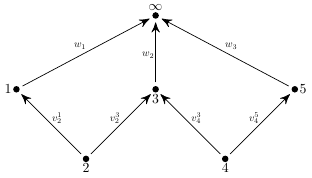} 

\caption*{general choice of weights}

\end{subfigure}
~
\begin{subfigure}[b]{0.5\textwidth}
\label{quiver2}
\centering
\includegraphics[scale=2.25]{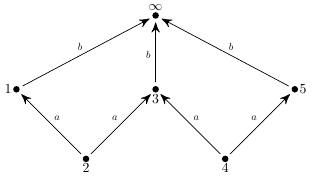}
\caption*{so-called democratic choice of weights}
\label{quiver2}

\end{subfigure}
\end{figure}


In the democratic case (on the right), the arrows in the Hasse quiver of $P^+$ which were actually in Hasse quiver for $P$ are labeled with one weight, while the "new" arrows are all labeled with a different value.  Of particular interest is when $(a,b) \in \mathbb{N} \times \mathbb{N}$ (see Remark \ref{weights} below).  We will confine our attention to this case in this paper (For an analysis of the non-democratic case, see \cite{meehan_meyer_2}).  When the Hasse quiver for $P^+$ is as above, we will say that $(a,b)$ is a weight.
\begin{definition}
Now, we let $d_{a,b}$ denote the weighted graph metric on the Hasse quiver of $P^+$, and let $\mathcal{D}$ be a category.  Then, $D = D(d_{a,b})$ is the interleaving metric corresponding to the weight $(a,b)$ on  ${\mathcal{D}}^P$.
\end{definition} 
With this modification, since any two generalized persistence modules can be interleaved, $D$ defines the structure of a finite metric space on the isomorphism classes of elements of ${\mathcal{D}}^P$.  We will now write $\mathcal{T}(P)$ for $\mathcal{T}((P^+)^-)$, and from this point forward, we suspend all posets at infinity.

\begin{rmk}
\label{weights}
Ultimately, we wish to consider a function defined on a full subcategory of isomorphism classes of $A(P)$-modules equipped with the interleaving distance $D=D(d_{a,b}).$  Of course, if the function $J$ in our workflow diagram takes values in a metric space $X$ with finite diameter, one can always choose $(a,b) = (\textrm{diam}(X), \textrm{diam}(X))$ to make the function a contraction.  Thus, in future work, we endow the weightspace $\mathbb{N} \times \mathbb{N}$ with the lexicographic ordering, and will consider \emph{minimal} weights $(a,b)$ such that the function in question is a contraction for $D=D(d_{a,b})$.
\end{rmk}

When the category $\mathcal{D}$ is $K$-modules, $D=D(d_{a,b})$ will be our interleaving distance on the category $\mathcal{C} \subseteq {\mathcal{D}}^P \cong A(P)$-mod.  We now endow the set of isomorphism classes of $A(P)$-modules with the other metric structure.

\subsection{Bottleneck Metrics}
\label{sec bottle}

A bottleneck metric provides an alternate metric structure on the set of isomorphsim classes of $A(P)$-modules, or indeed any subcategory $\mathcal{C}$ generated by a fixed collection of indecomposable modules.  The construction begins with a metric $d_2$ on a set $\Sigma$ , where $\Sigma$ is a subset of isomorphism classes of indecomposable $A(P)$-modules.  Additionally, we require a function $W:{\Sigma} \to (0,\infty)$, compatible with $d_2$ in the sense that for all $\sigma_1, \sigma_2 \in \Sigma$, 
$$|W(\sigma_1) - W(\sigma_2)| \leq d_2(\sigma_1,\sigma_2).$$ 

Following \cite{induced_matchings}, \cite{zigzag},  we define a matching between two multisets $S, T$ of ${\Sigma}$ to be a bijection $f:S' \to T'$ between multisubsets $S' \subseteq S$ and $T' \subseteq T$.  For $\epsilon \in (0,\infty)$, we say a matching $f$ is an $\epsilon$-matching if the following conditions hold;
\begin{enumerate}[(i)]
\item $\textrm{for all }s \in S, W(s) > \epsilon \implies s \in S'$ 
\item $\textrm{for all }t \in T, W(t) > \epsilon \implies t \in T' $, and
\item $d_2(s,f(s) ) \leq \epsilon$, for all $s \in S$.
\end{enumerate}

Since, intuitively $W$ measures the \emph{size} of an element of ${\Sigma}$, we call $W(\sigma)$ the \emph{width} of sigma.  Thus, in an $\epsilon$-matching, elements of $S$ and $T$ which are actually identified are within $\epsilon$, while all those not identified have width at most $\epsilon$.  

Given a $A(P)$-module $M$, the barcode of $M$, $B(M)$ is the multiset of the isomorphism classes of indecomposable summands of $M$ with their corresponding multiplicities.  Thus, $B(M)$ is precisely a multiset of elements in $\Sigma $, when $\Sigma$ is the set of all isomorphism classes of indecomposable modules.
\begin{definition}
Let $S, T$ be two finite multisubsets of any set $\Sigma$.  Suppose $d_2$ and $W$ are compatible.  Then the bottleneck distance between $S$ and $T$ is defined by, 
$$D_B(S,T) = \textrm{inf} \{\epsilon \in \mathbb{R}: \textrm{there exists and }\epsilon \textrm{-matching between  }S, T\} $$

\end{definition}
Let $\Sigma$ be any fixed subset of isomorphism classes of indecomposable $A(P)$-modules.  If $M, N$ are $A(P)$-modules with the property that every indecomposable summand of $M$ or $ N$ is isomorphic to an element of $\Sigma$, then we may identify $M, N$ with their barcodes $B(M), B(N)$, two multisubsets of $\Sigma$.  Then, set 
$$ D_B(M,N) := D_B(B(M),B(N)).$$

While there are many examples of bottleneck metrics in the literature, in this paper, we will choose $d_2$ to be the interleaving metric corresponding to the weight $(a,b)$ restricted to $\Sigma$, where $\Sigma$ is the set of convex modules.  Our width will be an algebraic analogue of the width of the support of a one-dimensional persistence module.  In the next subsection, we define our subcategory $\mathcal{C}$.

\subsection{The Category Generated by Convex Modules}
\label{subconvex}

Since a finite poset $P$ may have the property that $A(P)$-mod is of wild representation type, a characterization of all of the isomorphsim classes of its indecomposable modules might not be possible.  Moreover, an indecomposable module is not determined by its support (see Definition \ref{support}).  Let $\Omega$ denote the set of isomorphism classes of indecomposable $A(P)$-modules.  Clearly, the function 
$$\Omega \xrightarrow{Supp} \mathcal{P}(P)\textrm{ , which sends } M \xrightarrow{Supp} \textrm{Supp}(M)\textrm{, its support}$$

\noindent
may have infinite (and unknowable) domain, but always has finite range.  Motivated by one-dimensional persistent homology, we normalize taking the perspective that the width of an indecomposable should be determined only by its support.  We, therefore restrict out attention to the category $\mathcal{C}$ generated by an appropriate set $\Sigma$ of indecomposable thin modules.  A module is \emph{thin} if its dimension vector consists of only zeros and ones.  
\begin{definition}
\label{convex}
An indecomposable module $M$ is convex, if it thin, and if it is isomorphic to a module $M'$ where $M'$ satisfies
 $$\textrm{ for all } x,y \in \textrm{Supp}(M') \textrm{, with } x \leq y, \textrm{ the linear map } M'(x \leq y)\textrm{ is given by } Id_K.$$
\end{definition}

Let $\mathcal{C}$ be the full subcategory of $A(P)$-modules which are direct sums only of convex modules.  This is the full subcategory of $A(P)$-modules that we will focus on.  We note that in the literature, \emph{convex modules} are sometimes called \emph{interval modules} (see \cite{zigzag}).  We use convex instead to avoid confusion with either subsets of the poset  $P$, or elements of  its poset algebra $A(P)$.  In particular, some convex modules are supported in an honest interval in the poset, while others are not.

Clearly, when restricted to the set of isomorphism classes of convex modules, the function $M \to \textrm{Supp}(M)$ is one-to-one.  Of course, the function is not onto, as not every subset of $P$  is the support of a convex module.  One easily checks that if $S \subseteq P$, then there exists a convex module $M$ (unique up to isomorphism) with Supp$(M) = S$ if and only if 
\begin{enumerate}[(i)]
\item For all $s_1, s_2 \in S$ there exists an unoriented path in the Hasse quiver of $P$ that connects $s_1$ and $s_2$ staying entirely within $S$ , and
\item For all $s_1, s_2 \in S$ the set $\{p \in P: s_1 \leq p \leq s_2 \} = [s_1,s_2] \subseteq S$.
\end{enumerate}
In the above, an unoriented path is a product of paths and their formal inverses.  If $S$ satisfies (i), we say $S$ is \emph{connected}, and if $S$ satisfies (ii), we say $S$ is \emph{interval convex}.  Regardless of the representation type of the poset $P$, $\Sigma = \{[\sigma]: \sigma \textrm{ is convex }\}$ is finite.

It is well known that if $P$ has no \emph{crowns} (a subposet of a certain form), then every indecomposable thin $A(P)$-module is a convex module \cite{strongly_simply_connected}.  On the other hand, when $P$ has non-trivial cohomology, many indecomposable thin modules will not be convex (see Example \ref{cohomology} below from \cite{feinberg}).  While the class of posets we will restrict to in the next section contain many posets of wild representation type, they all have the property that every indecomposable thin is convex.

\begin{ex}
\label{cohomology}
Consider the posets given below.

\begin{figure}[H]
\centering
\begin{subfigure}[b]{0.5\textwidth}
\label{quiver1}
\centering
\begin{tikzpicture}[description/.style={fill=white,inner sep=2pt}]
\matrix (m) [matrix of math nodes, row sep=3em,
column sep=2.5em, text height=1.5ex, text depth=0.25ex]
{ {\bullet{}} & {\bullet{}} & {\bullet{}}\\ 
{\bullet{}} & {\bullet{}} & {\bullet{}}\\};
\node[above,scale=1] at (m-1-1) {$2$};
\node[above,scale=1] at (m-1-2) {$4$};
\node[above,scale=1] at (m-1-3) {$6$};
\node[below,scale=1] at (m-2-1) {$1$};
\node[below,scale=1] at (m-2-2) {$3$};
\node[below,scale=1] at (m-2-3) {$5$};
\path[->,font=\scriptsize]

(m-2-1) edge node[auto] {} (m-1-1)
(m-2-1) edge node[auto] {} (m-2-2)
(m-1-1) edge node[auto] {} (m-1-2)
(m-1-2) edge node[auto] {} (m-1-3)
(m-2-2) edge node[auto] {} (m-1-2)
(m-2-2) edge node[auto] {} (m-2-3)
(m-2-3) edge node[auto] {} (m-1-3);
\end{tikzpicture}

\caption*{F}

\end{subfigure}
~
\begin{subfigure}[b]{0.5\textwidth}
\label{quiver2}
\centering
\begin{tikzpicture}[description/.style={fill=white,inner sep=2pt}]
\matrix (m) [matrix of math nodes, row sep=3em,
column sep=2.5em, text height=1.5ex, text depth=0.25ex]
{ {\bullet{}} & \bullet{}\\ 
{\bullet{}} & \bullet{}\\};
\node[above,scale=1] at (m-1-1) {$3$};
\node[above,scale=1] at (m-1-2) {$4$};
\node[below,scale=1] at (m-2-1) {$1$};
\node[below,scale=1] at (m-2-2) {$2$};
\path[->,font=\scriptsize]

(m-2-1) edge node[auto] {} (m-1-1)
(m-2-1) edge node[auto] {} (m-1-2)
(m-2-2) edge node[auto] {} (m-1-1)
(m-2-2) edge node[auto] {} (m-1-2);
\end{tikzpicture}
\caption*{G}
\label{quiver2}

\end{subfigure}
\end{figure}

\noindent
The convex modules for the algebra with poset F have supports given by the following subsets;
\begin{eqnarray*}
&&\{ 1 \}, \{ 2 \}, \{ 3 \}, \{ 4 \}, \{1,2\} , \{1,3\}, \{2,4\}, \{3,4\}, \{3,5\}, \{4,6\}, \{5,6\}, \{1,2,3\}, \{1,3,5\},\{2,3,4\}, \\
&&\{2,4,6\}, \{3.4.5\},\{4,5,6\}, \{1,2,3,4\}, \{1,2,3,5\}, \{2,3,4,5\}, \{2,4,5,6\}, \{3,4,5,6\}, \{1,2,3,4,5\},\\
&& \{2,3,4,5,6\}, \{1,2,3,4,5,6\}.
\end{eqnarray*}
\noindent
Poset G has the property that when $K$ is infinite, there are infinitely many non-isomorphic indecomposable thin modules modules with full support.  Of course, there is exactly one convex module with full support.

\end{ex}
Convex modules are of interest in representation theory.  For example, in a large class of algebras, each simple modules can be associated to a collection (in fact, a poset) of convex modules in such a way that the representation type of the algebra can be determined in an effective fashion \cite{thin1}, \cite{thin2}, \cite{ringel}.  In persistent homology similar classes of generalized persistence moduels have frequently been used (see \cite{block}, \cite{zigzag}).  From the perspective of representation theory, it is easy to see that all of the indecomposable projectives and injectives, and all simples modules are convex.  Moreover, since convex modules are uniquely determined by their support, we agree with the sentiment in \cite{zigzag}, that $\mathcal{C}$ is the correct categorical framework for generalized persistence modules for $P$, when $P$ has arbitrary representation type.

\section{A Particular Class of Posets}
\label{posets}
In this section we confine our discussion to a certain class of finite posets.  Though easy to describe, most such posets are of wild representation type (see the discussion in Subsection \ref{algebra}).  We will restrict to $\mathcal{C}$, the full subcategory of $A(P)$-modules which are isomorphic to a direct sum of convex modules.

Let $P$ be a finite poset such that:
\begin{enumerate}
\item $P$ has a unique minimal element $m$,
\item for every maximal element $M_i \in P$, the interval $[m,M_i]$ is totally ordered, and
\item $[m,M_i] \cap [m,M_j] = \{m\}$ for all $i \neq j$.\\

\noindent
As a technical convenience, we sometimes also assume 
\item
their exists an $i_0$ with $\big| [m,M_{i_0}] \big| > \big| [m,M_i]\big|$, for all $i \neq i_0$.  

\end{enumerate}
That is, $P$ is a tree which brances only at the its unique minimal element and has one totally ordered segment longer than the others. 

\begin{definition}
If $P$ satisfies conditions (1), (2), (3), we say $P$ is an $n$-Vee, where $n$ denotes the number of maximal elements in $P$.  If, in addition, $P$ satisfies (4) we say that $P$ is an asymmetric $n$-Vee. 
\end{definition}

Clearly, a $1$-Vee is exactly a finite totally ordered set.  It is easy to see that every $1$-Vee is an asymmetric.  We will prove our isometry theorem for $n$-Vees.  
\begin{ex}
Poset D in Example \ref{Hasse} is an asymmetric $3$-Vee (with wild representation type).
\end{ex}

\begin{rmk}
\label{minimal}
The convex modules for $n$-Vees have some nice properties.  Note that if $P$ is any finite poset, then, the following two statements are equivalent:
\begin{enumerate}[(i)]
\item $P$ has a unique minimal element $m$ and every maximal interval in $P$, $[m,M_i]$ is totally ordered.
\item the support of every convex module has a unique minimal element. 
\end{enumerate}
That is to say, finite posets satisfying  only properties (1) and (2) in the definition for $n$-Vees are precisely those posets for which the support of a convex module always has a unique minimal element.  The proof is easy, but we include it.
\end{rmk}
\begin{proof}
First, if $P$ is as above, from the characterization of convex modules in Subsection \ref{subconvex} it is clear that the support of each convex module has a unique minimal element.  On the other hand, for a contradiction suppose $P$ satisfies (ii), but not (i).  Let $S \subseteq P$ denote the support of a potential convex module.  If $P$ has at least two minimals, then set $S$= $P$.  Thus it must be the case tht $P$ has a unique minimal $m$.  If there is a maximal interval $[m,M_j]$ contained in $P$ with $[m,M_j]$ not totally ordered.  Then, there exist $x, y \in [m,M_j]$ with $x, y$ not comparable.  But then $S = [x, M_j] \cup [y,M_j]$ is the support of a convex module contradicting (ii). 
\end{proof}

We will now establish some properties of the collection of translations of an asymmetric $n$-Vee.  Much (but not all) carries over to (general) $n$-Vees (see the end of the proof of Theorem \ref{main}).

\begin{lemma}
\label{T(P)}
Let $P$ be an asymmetric $n$-Vee, and let $(a,b)$ be any weights.  Let $d = d_{a,b}$ denote the weighted graph metric on the Hasse quiver of $P^+$ corresponding to $(a,b)$.  Then,
\begin{enumerate}[(i)]
\item For each $\epsilon \in \{ h(\Lambda): \Lambda  \in \mathcal{T}(P)\}$, the set $\{ \Gamma \in \mathcal{T}(P): h(\Gamma) = \epsilon \} $ has a unique maximal element ${\Lambda}_{\epsilon}$.
\item The set $\{ {\Lambda}_{\epsilon} \}$ is totally ordered, and ${\Lambda}_{\epsilon} \leq {\Lambda}_{\delta}$ if and only if $\epsilon \leq \delta$.
\item If $\Lambda, \Gamma \in \mathcal{T}(P)$ with $h(\Lambda), h(\Gamma) \leq \epsilon$ then there exists a ${\Lambda}_{\delta}$ with $\Lambda, \Gamma \leq {\Lambda}_{\delta}$, and $h({\Lambda}_{\delta})= \delta = \textrm{ max} \{ h(\Lambda), h(\Gamma) \}$.
\end{enumerate}

\end{lemma}

\begin{proof}
Let $P$ be as above.  First, say $n > 1$, then $P= \bigcup [m,M_i]$, with $[m,M_{i_0}]$ of maximal cardinality. Let $T_i = |[m,M_i]|-1$, so by hypothesis, $T_{i_0} > T_i$ for all $i \neq i_0$.  Let $T$ = max$\{T_i: i \neq i_0\}$ (note that if $P$ was not asymmetric $T = T_{i_0}$).  Let $\epsilon \in \{ h(\Lambda): \Lambda  \in \mathcal{T}(P)\}$ and suppose $h(\Lambda)=\epsilon$.  If $\Lambda m > m$, then $\epsilon \geq a T + b$, since;
\begin{eqnarray*}
&&\textrm{ if } \Lambda m = \infty \textrm{, then } h(\Lambda) = a T_{i_0} + b ,\\
&&\textrm{ if } \Lambda m \in (m,M_{i_0}] \textrm{, then } h(\Lambda) \geq a T + b\textrm{, and }\\
 &&\textrm{ if } \Lambda m \in (m,M_i], i \neq i_0 \textrm{, then } h(\Lambda) = a T_{i_0} + b.
 \end{eqnarray*}

Therefore, if $\epsilon < aT + b$, $\Lambda m = m$.  Then, $\Lambda \leq {\Lambda}_{\epsilon}$, where\\
${\Lambda}_{\epsilon} (x)= $
$\begin{cases}
m,\textrm{ if } x = m\\
\textrm{max} \{y \in (m,M_i]\cup \{\infty\}: d(x,y) \leq \epsilon\}, \textrm{ if } x \in (m,M_i]\\
\end{cases}$  

\medskip
\noindent
On the other hand, if $aT_{i_0} + b > \epsilon \geq aT+b$, then $\Lambda m \in [m,M_{i_0}]$, and $\Lambda((m,M_i])= \infty$ for $i \neq i_0$.  In this case, $\Lambda \leq {\Lambda}_{\epsilon}$, where\\
${\Lambda}_{\epsilon}(x) =$
$\begin{cases}
\infty, x \in (m,M_i], i \neq i_0\\
\textrm{max} \{y \in (m,M_{i_0}]\cup \{\infty\}: d(x,y) \leq \epsilon\}, \textrm{ if } x \in [m,M_{i_0}]\\
\end{cases}$

\medskip
\noindent
Lastly, if $\epsilon = aT_{i_0}+b$, then $\Lambda \leq \Lambda_{\epsilon}$, where $\Lambda_{\epsilon}(x) = \infty$, for all $x$.

Note that the formulae above are well defined, since $[m,M_i] \cap [m,M_j] = \{m\}$ for all $i \neq j$.  Now, suppose that $n = 1$.  Then $\Lambda \leq \Lambda_{\epsilon}$, where $\Lambda_{\epsilon}(x) = \textrm{max}\{y \geq x: d(x,y) \leq \epsilon\}$ for any $\epsilon$.  This proves (i).  The expressions for $\Lambda_{\epsilon}$ show that (ii) holds.  Now let $\Lambda, \Gamma \in \mathcal{T}(P)$ with $h(\Lambda), h(\Gamma) \leq \epsilon$, and suppose max$\{h(\Lambda), h(\Gamma)\} = \delta$.  Without loss of generality, say $h(\Lambda)=\delta, h(\Gamma) \leq \delta$.  Then, $\Lambda \leq {\Lambda}_{\delta}$ and $\Gamma \leq {\Lambda}_{h(\Gamma)} \leq {\Lambda}_{\delta}$, by (i), (ii) as required.

\end{proof}

The important observation is that although $\mathcal{T}(P)$ is not totally ordered, (for $n > 1$) it is directed in such a way that one may pass to a larger translation without increasing the height.  In contrast, for an arbitrary finite poset $P$,  $\mathcal{T}(P)$ will still be a directed set (because we suspended at infinity).  It may be the case, however, that for all ${\Lambda}_0$ with $ \Lambda, \Gamma \leq {\Lambda}_0$, $h({\Lambda}_0) > \kappa > \textrm{max}\{h(\Lambda),h(\Gamma)\}$.  That is to say, one may have to pay a price when passing to any larger common translation.  Lemma \ref{T(P)} shows that this does not happen for asymmetric $n$-Vees.  We are now ready to define the width of a convex module.

\begin{lemma}
\label{W}
Let $P$ be an asymmetric $n$-Vee, and let $(a, b)$ be a weight.  Then for all $I$ convex,  the following are equal;
\begin{enumerate}[(i)]
\item $W(I) = W_1(I) = \textrm{ min}\{\epsilon: \exists \Lambda, \Gamma \in \mathcal{T}(P), h(\Lambda), h(\Gamma) \leq \epsilon, \textrm{and Hom}(I,I\Lambda \Gamma) = 0\}$
\item $W_2(I) = \textrm{min}\{\epsilon: \exists \Lambda \in \mathcal{T}(P), h(\Lambda) \leq \epsilon, \textrm{ and Hom}(I,I {\Lambda}^2)=0\}$.
\item $W_3(I) = \textrm{min}\{\epsilon:\exists {\Lambda}_{\epsilon} \in \mathcal{T}(P) \textrm{with Hom}(I,I{{\Lambda}_{\epsilon}}^2)=0\}$.
\end{enumerate}
\end{lemma}

Before proving Lemma \ref{W}, we note that for any $I$ convex and for any $\theta \in \mathcal{T}(P)$, 
\begin{eqnarray*}
\textrm{Hom}(I,I\theta) \neq 0 \iff  \exists x \in \textrm{Supp}(I), \theta x \in \textrm{Supp}(I) \iff \theta x' \in \textrm{Supp}(I) \textrm{, for } x' \textrm{ minimal in Supp}(I).  
\end{eqnarray*}

This follows from general properties of module homomorphisms, and the observation in Remark \ref{minimal} that convex modules for $n$-Vees have unique minimal elements.  (See the Section \ref{homomorphisms and translations} for a detailed analysis of homomorphisms and translations)

Using this fact, we see that if $\Lambda \leq \Gamma$ and Hom$(I,I\Lambda) = 0$, then Hom$(I,I\Gamma)=0$.  Thus this condition defining $W$ produces an interval in $\{h(\Lambda): \Lambda \in \mathcal{T}(P)\}$.  We will now prove Lemma \ref{W}.
\begin{proof}
Let $\Lambda, \Gamma \in \mathcal{T}(P)$ with $h(\Lambda), h(\Gamma) \leq \epsilon$, and suppose Hom$(I,I\Lambda\Gamma) = 0$, and $\delta = \textrm{max}\{h(\Lambda), h(\Gamma)\}$.  Then, by Lemma \ref{T(P)}, there exists ${\Lambda}_{\delta}$, with $h({\Lambda}_{\delta})=\delta$ and $\Lambda, \Gamma \leq {\Lambda}_{\delta}$.  Then $\Lambda \Gamma \leq {{\Lambda}_{\delta}}^2$ so Hom$(I,I{{\Lambda}_{\delta}}^2)=0$, so $W_3(I) \leq W_2(I) \leq W(I)$.  But $S \subseteq T \implies \textrm{inf}(S) \geq \textrm{inf}(T)$, thus $W_3(I) \geq W_2(I) \geq W(I)$, so all are equal.  With this equivalence established, we define the width of a convex module.

\end{proof}

\begin{definition}
\label{defW}
Let $P$ be an asymmetric $n$-Vee and let $(a,b)$ be a weight.  Let $I$ be convex.  Then,
$$W(I) = W_1(I) = \textrm{ min}\{\epsilon: \exists \Lambda, \Gamma \in \mathcal{T}(P), h(\Lambda), h(\Gamma) \leq \epsilon, \textrm{and Hom}(I,I\Lambda \Gamma) = 0\}.$$
\end{definition}
While this definition of the width of a module is formulated algebraically, and is natural considering the structure of $\mathcal{T}(P)$, it is not without complication.  Intuitively $\frac{1}{2}|\textrm{Supp}(I)|$ (or perhaps $\lceil\frac{1}{2}|\textrm{Supp}(I)|\rceil$) is a first approximation of $W(I)$.  Indeed, this is the discrete analogue of the width used in the classical isometry theorem \cite{induced_matchings}, as their work corresponds to translations that are exactly constant shifts.  This discrete analogue of this is the choice of weights $(a,b) = (1,1)$ on a $1$-Vee.  For an $n$-Vee, however, modules with smaller support may happen to have large widths or the opposite.  For example, if $P$ is a $2$-Vee and $I$ is the simple convex module supported at $m$, then $W(I) = aT+b$.  In contrast, if $x \in (m,M_i)$ and $J$ is the convex module supported at $x, W(J) = a$.  Moreover, any convex module supported at $M_i$ for some $i$ necessarily has width greater or equal to $b$.  This is relevant, as no relation between $a$ and $b$ is specified.

The following Proposition will prove useful in Section \ref{totally} when we produce an explicit matching for $1$-Vees.  This result is an analogue of the corresponding statement in \cite{induced_matchings}.
\begin{prop}
\label{inj surj}
Let $P$ be an asymmetric $n$-Vee, let $A = {\bigoplus}_iA_i, C = \bigoplus_jC_j$ be in $\mathcal{C}$.  For any module $M$, let $B(M)$ denote the barcode of $M$ viewed as a multiset, and let $\Lambda \in \mathcal{T}(P)$.  
Then,
\begin{enumerate}[(i)]
\item If $A \xhookrightarrow{f} C$ is an injection, then for all $d \in P$, the set 
\begin{eqnarray*}
&&|\{i: [-,d] \textrm{ is a maximal totally ordered subset of Supp}(A_i)\}| \leq\\
&& |\{j: [-,d] \textrm{ is a maximal totally ordered subset of Supp}(C_j)\}|, \textrm{ and }
 \end{eqnarray*}
\item  If $A \xrightarrow{g} C$ is a surjection, then for all $b \in P$,
\begin{eqnarray*}
&&|\{j:[b,-] \textrm{ is a maximal totally ordered subset of Supp}(C_j) \}| \leq\\
&&|\{i:[b,-] \textrm{ is a maximal totally ordered subset of Supp}(A_i) \}|.
\end{eqnarray*}
\item If $A$ and $C$ are $(\Lambda,\Lambda)$-interleaved, and $A \xrightarrow{\phi} C\Lambda$ is one of the homomorphisms, then for all $I$ in $B(ker(\phi))$, $W(I) \leq h(\Lambda)$.
\item If $A$ and $C$ are $(\Lambda,\Lambda)$-interleaved, and $A \xrightarrow{\phi} C\Lambda$ is one of the homomorphisms, then for all $J$ in $B(cok(\phi))$, $W(J) \leq h(\Lambda)$.
\end{enumerate}
\end{prop}

Before proving the Proposition \ref{inj surj}, we state a Lemma.
\begin{lemma}
\label{quotient}
Let $P$ be an asymmetric $n$-Vee, say $P= \bigcup [m,M_i]$, with $[m,M_i]$ totally ordered.  Let $m_i = \textrm{ min}(m,M_i]$, and let ${\mathcal{I}}_j$ be the left ideal in $A(P)$ generated by $\{m_i: i \neq j\}$.
Then,
\begin{enumerate}[(i)]
\item For any $M$ convex, \\
$M/{\mathcal{I}}_j M$ is 
$\begin{cases}
0,\textrm{ if Supp}(M) \cap [m,M_j] = \phi\\
\textrm{the convex module with support given by } Supp(M) \cap [m,M_j] \textrm{ otherwise. }\\
\end{cases}$
\item For $A, B \in \mathcal{C}$, If $f$ is a homomorphism $A \xrightarrow{f} B/{\mathcal{I}}_j B$, then $f$ factors through $A/ {\mathcal{I}}_j A$.

\end{enumerate}
\end{lemma}
\begin{proof}

(i) obvious.  Statement (ii) is clear, since for $f:A \to B/{\mathcal{I}}_j B$, $w  \in {\mathcal{I}}_j $, $f(w \cdot a) = w\cdot f(a) = 0$.
\end{proof}

Note that if $n=1$, the left ideal ${\mathcal{I}}_i$ is identically zero, but the above is still true.  We now prove Proposition \ref{inj surj}.
\begin{proof}
Let $A, C$ be as above.  For all $i$, let $A_i = A(P) x_i$, $x_i \in $ Supp$(A_i)$, and let $[x_i,X_i]$ be a maximal connected totally ordered subset of Supp$(A_i)$ (We do not suppose $x_i=m$).  Similarly, let $y_j$ be such that $C_j = A(P) y_j$.  For $i, j$ let $f^{i}_j: A_{i} \to C_j$.  Now suppose $A \xrightarrow{f} C$ is an injection.  Fix $i_0$ and $[x_{i_0},X_{i_0}]$ be maximal contained in Supp$(A_{i_0})$.  Since $f^{i_0} = (f^{i_0}_j): A_{i_0} \to \bigoplus_j C_j$ is an inclusion, for any $ t \in [x_{i_0},X_{i_0}] $, there exists $j(t)$ with $f^{i_0}_{j(t)} \neq 0$.  Since $f^{i_0}_{j(t)}$ is a homomorphism, $f^{i_0}_{j(t)} \neq 0 \implies  f^{i_0}_{j(t)}(x_{i_0}) \neq 0$, and it is not that case that there exists $\ell > X_{i_0}$, with $\ell \in $Supp$(C_{j_0})$.  Set $j_0 = j(X_{i_0})$. Therefore, 
$$\{j: [x_{i_0},X_{i_0}] \subseteq \textrm{Supp}(C_j) \textrm{ and for all } \ell, \ell > X_{i_0} \implies \ell \notin \textrm{ Supp}(C_j)\} \neq \phi.$$
Now, for $d \in P$, let \\
\begin{eqnarray*}
j(d) &=& \{j: [-,d] \subseteq \textrm{Supp}(C_j), \ell \notin \textrm{Supp}(C_j) \textrm{ for } \ell > d\}, and \\
i(d)&=& \{i: [-,d] \subseteq \textrm{Supp}(A_i), \ell \notin \textrm{Supp}(A_I) \textrm{ for } \ell > d\}.
\end{eqnarray*}
Clearly, $i(d) \neq \phi \implies j(d) \neq \phi$.  Now, let $d \in P$, $d \neq m$ with $i(d) \neq \phi$.  Say $d \in (m,M_k]$.  Then,
\begin{eqnarray*}
\bigoplus\limits_{i \in i(d)}A_i / {\mathcal{I}}_k A_i &\hookrightarrow & \bigoplus\limits_{\substack{j \in j(d')\\ d' \leq d}}C_j/ {\mathcal{I}}_k  C_j \hookrightarrow C /  {\mathcal{I}}_k  C\implies\\
\bigoplus\limits_{i \in i(d)}(A_i/ {\mathcal{I}}_k A)(d) &\hookrightarrow & \bigoplus\limits_{\substack{j \in j(d') \\d' \leq d}}(C_j/ {\mathcal{I}}_k C_j)(d) = \bigoplus\limits_{j \in j(d)}(C_j/ {\mathcal{I}}_k C_j)(d)
\end{eqnarray*}
where the above inclusions are induced from $f$ and the inclusion of a submodule into a larger module respectively.  Thus, $|i(d)| \leq |j(d)|$.  If $d = m$, then,
\begin{eqnarray*}
\bigoplus\limits_{i \in i(m)}A_i \hookrightarrow  \bigoplus\limits_{j \in j(m)}C_j\hookrightarrow C \implies \bigoplus\limits_{i \in i(m)}A_i (m)\hookrightarrow  \bigoplus\limits_{j \in j(m)}C_j(m) \hookrightarrow C(m) ,
\end{eqnarray*}
so $|i(m)| \leq |j(m)|$.  This proves (i). The proof of (ii) is similar, though one inducts on the the cardinality of $\mathcal{S} = \{b: [b,-] \textrm{ is a maximal totally ordered subset of }A_i \textrm{ for some } A_i\}.$

Now we prove (iii).  For a contradiction, suppose there exists an $I \in B(ker(\phi))$ with Hom$(I,I{\Lambda}^2) \neq 0$.  But then the diagram below commutes.

\begin{center}
\begin{tikzpicture}[commutative diagrams/every diagram]
	\matrix[matrix of math nodes, name=m, commutative diagrams/every cell,row sep=.7cm,column sep=.45cm] {
	 \pgftransformscale{0.2}
		I & & I{\Lambda}^2 \\
		& \bigoplus\limits_j C_j & \\ };
		
		\path[commutative diagrams/.cd, every arrow, every label]
			(m-1-1) edge node {} (m-1-3)
			(m-1-1) edge node[swap] {${\phi}_{| I}$} (m-2-2)
			(m-2-2) edge node[swap] {$\psi \Lambda$} (m-1-3);

\end{tikzpicture}
\end{center}
Thus, $\psi \Lambda \phi (I) \neq 0$, a contradiction.  This proves (iii).  \\
Now let $J \in B(cok(\phi))$.  For a contradiction, suppose $W(J) > h(\Lambda)$.  But then there exists  $[x,X]$ a maximal subinterval in Supp$(J)$ with $\Lambda^2 x \leq X$.  Let $\{ b_x + im(\phi), ... b_X + im(\phi)\}$ be the corresponding basis elements for $J$.  But then there exists $j$ such that 
\begin{enumerate}
\item $[x,X] \subseteq $Supp$(C_j \Lambda)$, and
\item $C_j \Lambda (y) \notin im(\phi)$, for $x\leq y \leq X$.

\end{enumerate}

Then, $\Lambda x, \Lambda X \in $Supp$(C_j) \implies {\Lambda}^2\Lambda x = \Lambda {\Lambda}^2 x \leq \Lambda X$ which is in the support of $C_j$.  Therefore $W(C_j) \geq h(\Lambda)$.  But then, the following diagram commutes.
\begin{center}
\begin{tikzpicture}[commutative diagrams/every diagram]
	\matrix[matrix of math nodes, name=m, commutative diagrams/every cell,row sep=.7cm,column sep=.45cm] {
	 \pgftransformscale{0.2}
		C_j(\Lambda x) & & (C_j{\Lambda}^2)(\Lambda x)= (C_j \Lambda)({\Lambda}^2 x) \\
		& J\Lambda (\Lambda x) & \\ };
		
		\path[commutative diagrams/.cd, every arrow, every label]
			(m-1-1) edge node {} (m-1-3)
			(m-1-1) edge node[swap] {} (m-2-2)
			(m-2-2) edge node[swap] {} (m-1-3);

\end{tikzpicture}
\end{center}
But then $(C_j {\Lambda}^2)(\Lambda x) \in im(\phi \Lambda)(\Lambda x)$, a contradiction.  This proves (iv) and finishes the proof.
\end{proof}
The Example below shows that (i),(ii) in the Proposition \ref{inj surj} cannot be extended from maximal totally ordered intervals to convex subsets.  
\begin{ex}
Consider the $2$-Vee $[m,M_1] \cup [m,M_2]$, where $m < x < M_1$ and $m < y < z < M_2$.  Let $C_1$ be the convex module supported on $\{m, x, M_1\}$, and $C_2$ be the convex module supported on $\{m, y, z, M_2\}$.  Say $C_1$ has basis $\{e_m, e_x, e_{M_1}\}$ and $C_2$ has basis $\{f_m, f_ y, f_z, f_{M_2}\}$.  Then the submodule of $C_1 \oplus C_2$ with basis $\{e_m + f_m, e_x, e_{M_1}, f_y, f_z, f_{M_2}\}$ is isomorphic to the convex module with full support.  Thus, let $A_1$ be the convex module with full support.  Then, $A_1 \hookrightarrow C_1 \oplus C_2$, and while one can make the claim in the Proposition for each maximal totally ordered subset of the support of $A_1$ separately, one cannot do so simultaneously.
\end{ex}

In the next section we study homomorphisms and translations and their properties in $\mathcal{C}$.

\section{Homomorphisms and Translations}
\label{homomorphisms and translations}
In this section we investigate the relationship between homomorphisms and translations in the category $\mathcal{C}$.  In the interest of generality, we will relax our hypotheses on the poset $P$.  In this section, unless otherwise specified, $P$ is any finite poset.  The functions defined in Definitions \ref{trimleft}, \ref{trimright} are analogues of functions used by Bauer and Lesnick \cite{induced_matchings}.  In this context, however, they fail to preserve $W$, and may annihilate a convex module.

Note that if $S \subseteq P$ is non-empty and interval convex, then it canonically determines the isomorphism class of an element of $\mathcal{C}$ under the identification;
$$S \to  \bigoplus M_i \textrm{, where Supp}( M_i) \textrm{ is the }i\textrm{th connected component of } S .$$

We use this in the definition below.
\begin{definition}
\label{trimleft}
Let $P$ be any finite poset and $M$ be convex.  Say $\textrm{Supp} (M) = \bigcup\limits_i [a_i,b_i]$, where $[a_i,b_i]$ are maximal intervals in $\textrm{Supp} (M)$, and let $\Gamma \in \mathcal{T}(P)$.  Then,
$M^{+ \Gamma} $ is the element of $\mathcal{C}$  given by Supp$(M^{+ \Gamma}) = S = \bigcup\limits_i [\Gamma a_i,b_i]$.  
\end{definition}

That is, $M^{+ \Gamma}$ is the direct sum determined by $S = \bigcup\limits_i [\Gamma a_i,b_i]$.
(Note that if $\Gamma a_i \nleq bi$, then $[\Gamma a_i, b_i]$ is empty.)  One easily checks that,
\begin{enumerate}[(i)]
\item If $P$ is an $n$-Vee, $M^{+\Gamma}$ is convex, or $0$, and
\item For a general poset $P$,  $M^{+\Gamma}$ is a submodule of $M$.
\end{enumerate}
Moreover, for $i \in P$,  
\begin{eqnarray*}
M^{+\Gamma}(i)=\sum_x im(M(x\leq\Gamma x\leq i))=im(M(x_0\leq\Gamma x_0\leq i)) \textrm{ for any } x_0 \leq \Gamma x_0 \leq i.
\end{eqnarray*}
That is, $\theta = {\theta}_i \in M^{+\Gamma}(i) \implies \theta \in im(M(x_0\leq\Gamma x_0\leq i)) \textrm{ for any } x_0 \leq \Gamma x_0 \leq i.$  Now, for $M \in \mathcal{C}$ arbitrary, set 
$$M^{+\Gamma} = \bigoplus\limits_t M_t^{+\Gamma} \textrm{, where } M = \bigoplus\limits_t M_t. $$
If $P$ has the property that for all $i \in P$, $(-\infty, i]$ is totally ordered, one can still find $x = x(i)$ such that $M^{+\Gamma}(i)=im(M(x\leq\Gamma x\leq i)) $
is still valid.  Thus, in particular, the result holds for as $n$-Vee.  Note that if $(-\infty,i]$ is not totally ordered, then  $M^{+\Gamma}(i)=\sum_x im(M(x\leq\Gamma x\leq i))  $.  When $P$ is an $n$-Vee, we now make a dual definition.
\begin{definition}
\label{trimright}
Let $P$ be an $n$-Vee and let $M$ be a convex module.  Say Supp$(M) = \bigcup\limits_i[x,X_i]$ where each $x \leq X_i \leq M_i$.  (Recall that since $P$ is an $n$-Vee, the support of each convex module has a minimal element.  So $X_i \neq x$ for more than one $i$ implies $x = m$.)  Let $\Gamma \in \mathcal{T}(P)$.  Then, $M^{-\Gamma}$ is the convex modules with
$$\textrm{Supp}(M^{-\Gamma}) \textrm{ equal to } \bigcup\limits_i [x,X_i^{N(i)}] \textrm{, where } X_i^{N(i)} = \textrm{ max }\{y: \Gamma y\leq X_i, y \geq x\}.$$
Note that if no such $X_i^{N(i)}$ exists, $M^{-\Gamma} = 0$\\
\end{definition}

One easily checks that,
\begin{enumerate}[(i)]
\item If $M$ is convex, $M^{-\Gamma}$ is either identically zero or convex, and
\item $M^{-\Gamma}$ is a quotient of $M$.
\end{enumerate}

Because of (i) we may extend our definition from $\Sigma$ to $\mathcal{C}$.  For $M \in \mathcal{C}$, set
$$M^{-\Gamma}=\bigoplus\limits_t {M_t}^{-\Gamma} \textrm{, where } M = \bigoplus\limits_t M_t .$$

Notice that the assignment $M \to M^{+\Gamma}$ moves the left endpoints of the support of a module to the right, while $M \to M^{-\Gamma}$ moves the right endpoints of the support to the left.  A physical characterization of $M^{-\Gamma}$ is possible, but will prove unecessary for our purposes.

We now prove a useful proposition.

\begin{prop}
\label{trim}
Let $P$ be an $n$-Vee, and let $I, M \in \mathcal{C}$.  Let $(\phi, \psi)$ be a $(\Lambda, \Lambda)$-interleaving betweem $I$ and $M$.  Say $\phi :I\rightarrow M\Lambda$.  Then, 
\begin{enumerate}[(i)]
\item $I^{-\Lambda^2}$ is a quotient of both $I$ and $im(\phi)$, and
\item $M^{+\Lambda^2}\Lambda$ is a submodule of both $M\Lambda$ and $im(\phi)$.
\end{enumerate}
\end{prop}
\begin{proof}
First, by the comments above, $I^{-\Lambda^2}$ is a quotient of $I$.  Now, since $I$ and $M$ are $(\Lambda,\Lambda)$-interleaved, $\psi\Lambda\circ\phi=(I \to I\Lambda^2)$. Therefore, $$(\psi\Lambda)(im(\phi))=I^{-\Lambda^2},$$ and hence $I^{-\Lambda^2}$ is a homomorphic image, and hence a quotient of $im(\phi)$.  This proves (i).

We now prove (ii).  First, already $M^{+\Gamma}$ is a submodule of $M$ . Moreover, $C\leq D$ implies $\tau$, $C\tau\leq D\tau$, for any $\tau \in \mathcal{T}(P)$.   Hence, $M^{+\Lambda^2}\Lambda$ is a submodule of $M\Lambda$.  It remains to show that $M^{+\Lambda^2}\Lambda$ is a submodule of $im(\phi)$. Let $i\in\mathcal{P}$, $i\in\mathrm{Supp}((M^{+\Lambda^2})\Lambda)$.  Then, as a vector spaces, 
$$((M^{+\Lambda^2})\Lambda)(i)\subset(im(\phi))(i)\oplus_K(coker(\phi))(i).$$

Let $\theta_i\in((M^{+\Lambda^2})\Lambda)(i)$. Note that $\theta_i=\theta'_{\Lambda i}$. At the $i$ level, $\theta_i=a_i+b_i$ with $a_i\in(im(\phi))(i)$ and $b_i\in(coker(\phi))(i)$.  Then,
\begin{eqnarray*}
&&{\theta_i} \in im((M(x\leq\Lambda^2x\leq\Lambda i)))\implies\\
&&\theta'_{\Lambda i} =M(x\leq\Lambda^2x\leq\Lambda i)(a_x+b_x), a_x \in im(\phi), b_x \in coker(\phi) \implies \\
&&M(x\leq\Lambda^2x\leq\Lambda i)(a_x)=a_i+\alpha_i, \textrm{ with }\alpha_i\in im(\phi), \textrm{ and }\\
&&M(x\leq\Lambda^2x\leq\Lambda i)(b_x)=-\alpha_i+b_i.
\end{eqnarray*}
But, by Proposition \ref{inj surj}, $W(coker(\phi))<h(\Lambda)$, thus $-\alpha_i+b_i=0$, and therefore $\alpha_i=b_i=0$. Hence, $\theta_i$ was fully contained in $(im(\phi))(i)$.

Thus, 
\begin{eqnarray*}
&&im((M(x\leq\Lambda^2x\leq\Lambda i)))\subset(im(\phi))(i) \textrm{ for all }i \textrm{, hence }\\
&&((M^{+\Lambda^2})\Lambda(i)\subset(im(\phi)))(i) \textrm{ for all }i.\\
\end{eqnarray*}
Therefore, $$(M^{+\Lambda^2})\Lambda\leq im(\phi).$$
This proves (ii).
\end{proof}

We will now consider the action of $\mathcal{T}(P)$ on $\mathcal{C} \cup \{0\}$.  We first point out that, in general, the monoid $\mathcal{T}(P)$ need not act on $\Sigma \cup \{0\}$.

\begin{ex}
\label{diamond}

Let $P$ be the poset E in Example \ref{Hasse}, and let $\Lambda$ be the translation $1 \to 1, 2 \to 4, 3 \to 4, 4 \to \infty, \infty \to \infty $.  Let $J$ be the convex module with support equal to $\{2, 3, 4\}$.  Then, $J \Lambda \cong S \oplus T$, where $S$ is the simple supported on $\{2\}$, and $T$ is the simple supported on $\{3\}$.  Alternatively, let $P$ be the poset $1, 2 \leq 3$, with $1, 2$ not comparable.  Let $J$ be the convex module with full support, and $\Lambda$ be given by $1 \to 1, 2 \to 2, 3 \to \infty, \infty \to \infty$.  Then, $J \Lambda$ is again a direct sum of two convex modules.
\end{ex}

Example \ref{diamond} shows that the action of $\mathcal{T}(P)$ on $\mathcal{C} \cup \{0\}$ need not restrict to $\Sigma \cup \{0\}$.  In Lemma \ref{hom1} we will see that when $P$ is an $n$-Vee, however, the action does restrict.  First, a quick observation.

\begin{lemma}
\label{unique min}
Let $P$ be any poset with a unique minimal element $m$, and suppose $\Lambda \in \mathcal{T}(P)$ with $\Lambda m = m$.  Then for all convex $J$ with $m \in $ Supp$(J)$, $J \Lambda$ is convex.  
\end{lemma}
\begin{proof}
Let $P$ be as above, $M$ be convex, with $m \in $ Supp$(M)$.  Let $\Lambda$ be a translation with $\Lambda m = m$.  Clearly $M \Lambda$ is thin.  Let $t_1, t_2 $ be in the support of $M \Lambda$, and suppose $t_1 \leq t \leq t_2$.  Then, $\Lambda t_1, \Lambda t_2 \in $ Supp$(M) \implies [\Lambda t_1, \Lambda t_2] \subseteq $Supp$(M)$, since $M$ is convex.  Since $\Lambda t \in [\Lambda t_1, \Lambda t_2]$, $t$ is in the support of $M$, so $[t_1, t_2] \subseteq $ Supp$(M \Lambda)$.  Now, since $\Lambda m = m$, $m \leq x$ for all $x$, Supp$(M \Lambda)$ is connected.
\end{proof}

\begin{lemma}
\label{hom1}
Let $P$ be an $n$-Vee, $I$ a convex module and $\Lambda \in \mathcal{T}(P)$.  Then, $I \Lambda$ is either the zero module or convex.
\end{lemma}
\begin{proof}
First, from the proof of Lemma \ref{unique min}, if non-zero $I \Lambda$ is in $\mathcal{C}$.  We now proceed in cases.  First, suppose $m \in $ Supp$(I)$.  If $\Lambda m = m$, then $m $ is in the support of $I \Lambda$, so $I \Lambda $ is convex.  On the other hand, if $\Lambda m \in (m,M_i]$, then for all $j  \neq i, (m,M_j] \cap $ Supp$(I \Lambda) = \phi$.  But then Supp$(I \Lambda) \subseteq [m, M_i]$, hence it is convex or zero, since it is interval convex.   If $m \notin $ Supp$(I)$, then $I$ is supported in $(m, M_j]$ for some $j$ and the result follows.

\end{proof}
  
We will now work towards the characterization of homomorphism between convex modules when $P$ is an $n$-Vee.  In the interest of generality, we begin with an arbitrary finite poset $P$.  
\begin{definition}
\label{capital phi}
Let $I, M$ be convex.  Let $\{e_x: x \in \textrm{Supp}(I)\}$, $\{f_x: x \in \textrm{Supp}(M)\}$ be $K$-bases for $I, M$ respectively.  Consider the linear function ${\Phi}_{I,M}$, defined by 
\begin{center}
${\Phi}_{I,M}(e_y) =
\begin{cases}
f_y, \textrm{ if } y \in \textrm{Supp}(I) \cap \textrm{Supp}(M)\\
0 \textrm{ otherwise.}

\end{cases}$
\end{center}
\end{definition}

By inspection, ${\Phi}_{I,M}$ is a non-zero module homomorphism if and only if Supp$(I) \cap \textrm{Supp}(M)$ satisfies, 
\begin{enumerate}[(i)]
\item Supp$(I) \cap \textrm{ Supp}(M) \neq \phi$
\item $x \in $ Supp$(I) \cap $ Supp$(M)$, $ y \geq x, y \in \textrm{ Supp}(M) \implies y \in \textrm{ Supp}(I)$, and 
\item $x \in $ Supp$(I) \cap $ Supp$(M)$, $ y \leq x, y \in \textrm{ Supp}(I) \implies y \in \textrm{ Supp}(M)$.
\end{enumerate}
 
Note that even when it is not a module homomorphsim, ${\Phi}_{I,M}$ can be viewed as the linear extension of ${\chi}(\textrm{Supp}(I) \cap \textrm{Supp}(M))$, the characteristic function on the intersection of the supports of $I$ and $M$.  

The following two lemmas will allow us to conclude that when $P$ is an $n$-Vee, up to a $K$-scalar, this is the only possible module homomorphisms from $I$ to $M$.

\begin{lemma}
\label{hom1.5}
Let $P$ be any finite poset, and let $I, M$ be convex.  Let $S  \subseteq \textrm{Supp}(I) \cap \textrm{Supp}(M)$, with $S$ nonempty.  Suppose that there exists an $N \in \mathcal{C}$ with Supp$(N) = S$.  Then, $N$ is isomorphic to the image of a non-zero module homomorphsim from $I$ to $M$ if and only if
\begin{enumerate}[(a)]
\item for all $x \in S$, if $y \in \textrm{Supp}(I)$ with $y \leq x$, then $y \in S$, and
\item for all $x \in S$, if $y \in \textrm{Supp}(M)$ with $y \geq x$, then $y \in S$.
\end{enumerate} 
\end{lemma}
\begin{proof}
$S$ corresponds to the support of a non-zero quotient module of $I$ if and only if $S$ satisfies (a).  Similarly, $S$ corresponds to a non-zero submodule of $M$ if and only if $S$ satisfies (b).  Since any homomorphism can be factored into an injection after a surjection, the result follows.
\end{proof}
 
\begin{lemma}
\label{hom2}
Let $P$ be an $n$-Vee.  Let $I, M$ be convex modules.  Then, Hom(I,M) $\cong K$ or $0$ (as a vector space)

\end{lemma}
\begin{proof}
First, let $P$ be any finite poset, and $I, M$ be convex.  Suppose that $g$ is any non-zero homomorpism from $I$ to $M$.  Then, by Lemma \ref{hom1.5},  $im(g) = I/ker(g)$ has support equal to $S \subseteq \textrm{Supp}(I) \cap \textrm{Supp}(M)$ satisfying (a), (b) from Lemma \ref{hom1.5}.  We will show that any such $S$ is a union of connected components of $\textrm{Supp}(I) \cap \textrm{Supp}(M)$.  Since $g$ is non-zero, $S$ is non-empty.  Now, let $s \in S$ and suppose that $y \in \textrm{Supp}(I) \cap \textrm{Supp}(M)$ with $y \geq s$.  Then, by (b), $y \in S$.  Similarly, if $y \in \textrm{Supp}(I) \cap \textrm{Supp}(M)$ with $y \leq s$, then by (a), $y \in S$.  Therefore $S$ contains the connected component of $s$ in $\textrm{Supp}(I) \cap \textrm{Supp}(M)$.  The result follows.

Now, if $P$ is an $n$-Vee, $\textrm{Supp}(I) \cap \textrm{Supp}(M)$ is connected, so $S$ must be the full intersection.  As above, let $\{e_x: x \in \textrm{Supp}(I)\}$, $\{f_x: x \in \textrm{Supp}(M)\}$ be $K$ bases for $I, M$ respectively.   

Then,
\begin{center}
$g(e_z) =
\begin{cases}
c_z f_z, \textrm{ if } z \in \textrm{Supp}(I) \cap \textrm{Supp}(M), c_z \in K\\
0 \textrm{ otherwise,}

\end{cases}$
\end{center}
 where Supp$(I)$, Supp$(M)$ satisfy conditions (i), (ii) and (iii) from below Definition \ref{capital phi}.  Clearly, every non-zero quotient of $I$ must have support containing the minimal element $t$ of Supp$(I) $.  Then, since $I, M$ are convex, 
$$g(I(t \leq y))e_t = g e_y = M(t \leq y)g e_t = M(t \leq y)c_t f_t = c_t f_y = c_y f_y.$$  Therefore, $g = c_t {\Phi}_{I,M}$.  Of course, if $g$ is identically zero, g is still in the span of ${\Phi}_{I,M}$.
\end{proof}

We now investigate the action of $\mathcal{T}(P)$ on Hom$(I,M)$, when $I, M$ are convex.  From the observation in the proof of Lemma \ref{hom2}, we see that for $P$ arbitrary, Hom$(I,M)$ will have dimension equal to the number of connected components of Supp$(I) \cap \textrm{Supp}(M)$.  Still, for a fixed translation $\Lambda \in \mathcal{T}(P)$, Hom$(I\Lambda,M\Lambda)$ may be trivial (even when $I\Lambda , M\Lambda$ are non-zero).  We now state a condition which ensures that Hom$(I,M)\cdot \Lambda \neq 0$.  This can be done more generally, but we state the result only for $P$ an $n$-Vee.

\begin{lemma}
\label{hom3}
Let $P$ be an $n$-Vee, and let $I, M$ be convex.  Let $\Lambda \in \mathcal{T}(P)$.  Say Hom(I,M) $\neq 0$ and there exists $t$ with $\Lambda t \in Supp(I) \cap Supp(M)$.  Then Hom$(I\Lambda,M\Lambda) \neq 0$.
\end{lemma}
\begin{proof}
Since $\Lambda t \in Supp(I) \cap Supp(M) $, $I \Lambda$ and $M\Lambda$ are not zero.  Also, by Lemma \ref{hom2}, Supp$(I)$ and Supp$(M)$ satisfy
the conditions (i), (ii) and (iii) from below Definiton \ref{capital phi}.  Since $I \Lambda, M\Lambda$ are convex, it is enough to show that the above still holds for $\textrm{Supp}(I\Lambda) \textrm{ and }\textrm{Supp}(M\Lambda)$.  Again, $t \in  \textrm{Supp}(I\Lambda) \cap \textrm{Supp}(M\Lambda)$, hence the intersection is nonempty.  Now let $z \in \textrm{Supp}(I\Lambda) \cap \textrm{Supp}(M\Lambda)$, with $w \in \textrm{Supp}(M\Lambda), w \geq z$.  Then, $\Lambda z \in \textrm{Supp}(I) \cap \textrm{Supp}(M), \Lambda w \in \textrm{Supp}(M)$, and $\Lambda w \geq \Lambda z$.  Therefore, $\Lambda z \in \textrm{Supp}(I)$, so $z \in $ Supp$(I \Lambda)$.  The last requirement is proved similarly.
\end{proof}

Note that conditions (ii) and (iii) are clearly inherited from $I, M$.  The authors point out that the hypothesis above that the intersection of supports coincides with the image of the translation is required even on a totally ordered set (see Example \ref{image of translation} below).
\begin{ex}
\label{image of translation}
Let $P$ be the totally ordered set $\{1, 2, 3, 4, 5, 6\}$ with its standard ordering, and let $\Lambda$ send $1 \to 2, 2 \to 3, 3 \to 3, 4 \to 5, 5 \to 6, 6 \to 6$.  Let $I$ and $M$ be the convex modules supported on $\{4, 5, 6\}$ and $\{3, 4\}$ respectively.  Note that Hom$(I,M) \neq $, $I\Lambda$ and $M\Lambda$ are supported on $\{4, 5\}$ and $\{2, 3\}$ respectively.  Clearly, the supports of $I \Lambda$ and $M \Lambda$ are disjoint, so Hom$(I\Lambda,M\Lambda) =0$.
\end{ex}

\begin{notation}
Let $P$ be an $n$-Vee, $\Lambda \in \mathcal{T}(P)$ and say $I$ is convex.  We write ${\Phi}_I^{\Lambda}$ for ${\Phi}_{I,I\Lambda}$, as $I\Lambda$ is either zero or convex.  For $I \in \mathcal{C}$, we write ${\Phi}_I^{\Lambda}$ for the canonical homomorphism as well, since it is necessarily diagonal.  (Of course, as mentioned above, even if $I \Lambda$ is not trivial, it may be the case that ${\Phi}_I^{\Lambda}$is identically zero.)
\end{notation}
We can now show that when $P$ is an $n$-Vee the collection of interleavings between two elements of $\mathcal{C}$ will have the structure of an affine variety (not necessarily irreducible).  Though the result still holds for more general posets, our proof is an application of the results of this section.  Some examples are provided in Section \ref{section ex}

\begin{prop}
\label{variety}
Let $P$ be an $n$-Vee and let $I = \bigoplus I_s, M = \bigoplus M_t$ be two elements of $\mathcal{C}$.  Let $\Lambda, \Gamma \in \mathcal{T}(P)$.  Then the collection of $(\Lambda, \Gamma)$-interleavings between $I$ and $M$ has the structure of an affine variety. 
\end{prop}

Indeed, as stated above the result holds for any finite poset, though when $P$ is an $n$-Vee the variety has a simpler description.  We sketch the proof.  Let $P, I, M, \Lambda$ be as above, and let $\phi, \psi$ be any interleaving between $I$ and $M$.  Thus, we obtain the commutative triangles below.

\begin{center}

\begin{tikzpicture}[commutative diagrams/every diagram]
	\matrix[matrix of math nodes, name=m, commutative diagrams/every cell,row sep=.7cm,column sep=.45cm] {
	 \pgftransformscale{0.2}
		I & & I\Gamma\Lambda  & & I\Gamma & \\
		& M\Lambda  & & M & & M\Lambda \Gamma\\ };
		
		\path[commutative diagrams/.cd, every arrow, every label]
			(m-1-1) edge node {${\Phi}_I^{\Gamma \Lambda}$} (m-1-3)
			(m-1-1) edge node[swap] {$\phi$} (m-2-2)
			(m-2-2) edge node[swap] {$\psi \Lambda $} (m-1-3)
			(m-2-4) edge node {$\psi$} (m-1-5)
			(m-2-4) edge node[swap] {${\Phi}_M^{\Lambda \Gamma}$} (m-2-6)
			(m-1-5) edge node {$\phi \Gamma$} (m-2-6);
					
\end{tikzpicture}

\end{center}

Therefore, as matrices of module homomorphisms;
\begin{eqnarray*}
[{\psi}^t_s \Lambda ] \cdot [{\phi}^s_t] = [{\Phi}_{I_s}^{\Gamma \Lambda}] \textrm{, and } [{\phi}^s_t \Gamma] \cdot [{\psi}^t_s] = [{\Phi}_{M_t}^{\Lambda \Gamma}]
\end{eqnarray*}
where $\phi$, $\psi$ decompose into their component homomorphisms $\phi^s_t : I_s \to M_t \Lambda$ and $\psi^t_s : M_t \to I_s \Gamma$ respectively.  By Lemma \ref{hom2}, $\phi_t^s, \psi_s^t$ are in the span of ${\Phi}_{I_s, M_t \Lambda}$ and ${\Phi}_{M_t, I_s \Gamma}$ respectively.  Hence, if Hom$(I_s, M_t \Lambda)$ is not identically zero, $\phi_t^s = \lambda^s_t {\Phi}_{I_s, M_t\Lambda}$, where $\lambda^s_t \in K$, with a similar result holding for $\psi^t_s$.  In addition, $(\lambda {\Phi}_{A,B}) {\Lambda}_0 = \lambda ({\Phi}_{A \Lambda_0, B \Lambda_0})$ for all scalars $\lambda$, translations $\Lambda_0$, and all $A, B$ convex.

Therefore, the interleavings between $I$ and $M$ correspond to the algebraic set given by values of $\lambda^s_t, \mu^t_s$ satisfying all quadratic relations obtained by evaluating the matrix equations above at all elements of $P$.  

More precisely, first suppose Hom$(I_s, M_t \Lambda), \textrm{Hom}(M_t,I_s\Lambda) = 0$ for all $s, t$.  In this case, the variety of interleavings $V^{\Lambda,\Gamma}(I,M)$ is given by\\

\medskip
$V^{\Lambda,\Gamma}(I,M) = \begin{cases} \textrm{the zero variety, if }W(I_s), W(M_t) \leq \textrm{max}\{h(\Lambda), h(\Gamma)\}\textrm{ for all }s, t\\
\textrm{the empty variety, otherwise}.\\
\end{cases}$

\medskip
The above cases correspond to whether or not setting all morphisms identically equal to zero corresponds to an admissible interleaving between $I$ and $M$.  On the other hand, suppose some of the relevant spaces of homomorphisms above are non-zero.
Then, let $r^s_t, q^t_s$ be given by
$$r^s_t = \lambda^s_t \cdot {dim}_K(\textrm{Hom}(I_s, M_t\Lambda)) \textrm{, and } q^t_s = \mu^t_s \cdot {dim}_K(\textrm{Hom}(M_t,I_s\Lambda)).$$

Also, let 
$$\bar{r}^s_t = r^s_t \cdot {dim}_K(\textrm{Hom}(I_s\Lambda, M_t \Lambda^2)) \textrm{, and }\bar{q}^t_s = q^t_s \cdot {dim}_K(\textrm{Hom}(M_t\Lambda,I_s \Lambda^2)).$$  

Let $R$ denote the $|T| \times |S|$ matrix $R = [r^s_t{\Phi}_{I_s, M_t\Lambda}]$.  Similarly, let $Q$ denote the $|S| \times |T|$ matrix $Q = [q^t_s{\Phi}_{M_t,I_s\Lambda}]$.  Also, set $\bar{R} = [\bar{r}^s_t{\Phi}_{I_s\Lambda, M_t \Lambda^2}], \bar{Q} = [\bar{q}^t_s{\Phi}_{M_t\Lambda,I_s\Lambda^2}]$. 

\vspace{.2 in}
Then, since $\phi^s_t = \lambda^s_t {\Phi}_{I_s,M_t\Lambda}$, and $\psi^t_s = \mu^t_s {\Phi}_{M_t,I_s\Lambda}$, the homomorphisms $\phi$, and $\psi$ correspond to an interleaving if and only if the equations below are satisfied, when evaluated at all elements of the poset $P$.
\begin{eqnarray}
\bar{Q} \cdot R = [\bar{q}^t_s{\Phi}_{M_t\Lambda,I_s\Lambda^2}] \cdot [r^s_t{\Phi}_{I_s, M_t\Lambda}] = [{\Phi}_{I_s}^{\Gamma \Lambda}], \bar{R} \cdot Q = [\bar{r}^s_t{\Phi}_{I_s\Lambda, M_t \Lambda^2}] \cdot [q^t_s{\Phi}_{M_t\Lambda,I_s\Lambda^2}] = [{\Phi}_{M_t}^{\Lambda \Gamma}]
\end{eqnarray}

\vspace{.2 in}
Therefore, in this situation $V^{\Lambda,\Gamma}(I,M)$ is the affine algebraic set with coordinate ring given by $K[\{\lambda^s_t: \textrm{Hom}(I_s,M_t \Lambda ) \neq 0\},\{\mu^t_s: \textrm{Hom}(M_t, I_s \Gamma) \neq 0\}]$ modulo the ideal given by all identities from (1).  For some computations, see Examples \ref{ex 4}, \ref{ex 3}.

When $P$ is not an $n$-Vee (or at least a tree branching only at a unique minimal element), the collection of interleavings still admits the structure of a variety, though the description is more cumbersome.

\begin{rmk}
\label{variety distance}
Using Proposition \ref{variety}, we may visualize the interleaving distance between two elements of $\mathcal{C}$ as follows.  Let $(a,b)$ be any weight, and let $I, M \in \mathcal{C}$.  For each $\epsilon \in \{h(\Lambda)\}$, let $V_{\epsilon}(I,M)$ denote the variety of $(\Lambda_{\epsilon},\Lambda_{\epsilon})$-interleavings between $I$ and $M$.  Then, 
$$D(I,M) = \textrm{min}\{\epsilon:  \textrm{ the variety }V^{\Lambda_{\epsilon},\Lambda_{\epsilon} }(I,M) \textrm{ is non-empty}\} .$$

For some computations see Example \ref{ex new}.
\end{rmk}

We now observe that our width gives rise to a bottleneck metric when $P$ is an $n$-Vee.  For this $W$ must be compatible with the interleaving distance in the sense of Subsection \ref{sec bottle}.

\begin{prop}
\label{W and d}
Let $P$ be an $n$-Vee, and let $(a,b)$ be weights.  Let $D = D(d_{a,b})$ be the interleaving distance, and $W$ be the width function.  Then, for $I,M$ convex,
 $$|W(I) - W(J)| \leq D(I,J).$$

\end{prop}
The proof, which proceeds in cases, is omitted.  Since $W$ and $D$ are compatible on $\Sigma$, we obtain a bottleneck metric on the category $\mathcal{C}$ (see Subsection \ref{sec bottle}).   Let $D_B$ denote this bottleneck metric.  In the next section we will prove an isometry theorem for $1$-Vees.


\section{Isometry Theorem for Finite Totally Ordered Sets}
\label{totally}

We now prove the isometry theorem for finite totally ordered sets.  We will fix notation in this section for our poset.  Let $P = \{m < m_1 < m_2 < ... < n = M_1\} = [m,n]$ be totally ordered (a $1$-Vee), and fix any weight $(a,b)$.  Note that, in this section only, $n$ does not correspond to the number of maximal elements in $P$.  We begin with some preliminary observations.

\begin{lemma}
\label{hom4}
Let $P = \{m < m_1 < m_2 < ... < n = M_1\} = [m,n]$, and suppose $\Lambda$ be a power of a maximal translation with given height. Then,
\begin{enumerate}[(i)]
\item $im(\Lambda) \cap P =[\Lambda(m),n]$.
\item If $i\in [\Lambda(m),n)$, then ${\Lambda}^{-1}(i)$ is a singleton.
\item $\Lambda i=\Lambda j \in P \implies i=j$ or $\Lambda i=\Lambda j=n$.
\end{enumerate}
\end{lemma}
The result follows from the form of the maximal translation $\Lambda$ (see the proof of Lemma \ref{T(P)}).  Note that the power of a maximal translation need not be maximal.  Moreover, $h(\Lambda^2)$ need not be $2h(\Lambda)$.  The following Lemma follows from our characterization of the homomorphisms between convex modules in the last section (see Lemma \ref{hom2}).
\begin{lemma}
\label{hominterval}
If $I,J$ are convex modules for $P = \{m < m_1 < m_2 < ... < n = M_1\} = [m,n]$, then $\mathrm{Hom}(I,J)\neq0$ if and only if the endpoints of Supp$(I)=[x,X]$ and Supp$(J)=[y,Y]$ satisfy $$y\leq x\leq Y\leq X.$$ 
\end{lemma}
As previously mentioned, any homomorphism is a scalar in $K$ times ${\Phi}_{I,J}$ (see Definition \ref{capital phi}).

\begin{lemma}
Let $P$ be as above, and suppose $\Lambda = {\Lambda}_{\epsilon}$ is a maximal translation.  Let $A$ and $B$ be convex, and suppose $A\Lambda,B\Lambda\neq0$ and $\mathrm{Hom}(A,B)\neq0$.  Then $\mathrm{Hom}(A\Lambda,B\Lambda)\neq0$.
\end{lemma}

\begin{proof}
Let $s \in $Supp$(A) \cap \textrm{ Supp}(B)$.  If $s \in im(\Lambda)$ we are done by Lemma \ref{hom3}.  Otherwise, $[x,Y]  \cap im(\Lambda)$ is empty, where Supp$(A) = [x,X]$, Supp$(B) = [y,Y]$ with $y \leq x \leq Y \leq X$ as in Lemma \ref{hominterval}.  But then, by Lemma \ref{hom4}, $[y,Y]$ is disjoint from the image of $\Lambda$, therefore, $B \Lambda = 0$, a contradiction.

\end{proof}
Lastly, the following is an easy consequences of the results of the previous section (see Definitions \ref{trimleft}, \ref{trimright}).
\begin{lemma}
\label{prime}
Let $P$ be totally ordered.  Then the following are equivalent:
\begin{enumerate}[(i)]
\item $\mathrm{Hom}(J,J\Lambda^2)\neq0$
\item there is an $ x\in\mathrm{Supp}(J)$ with $\Lambda^2x\in\mathrm{Supp}(J)$
\item $J^{+\Lambda^2}\neq0$
\item $(J^{+\Lambda^2})\Lambda\neq0$
\item $J^{- \Lambda^2} \neq 0$.
\end{enumerate}
\end{lemma}
We are now ready to prove that every interleaving induces a matching of barcodes when $P$ is a $1$-Vee.  This is very much an algebraic reformulation of the results of Bauer and Lesnick in \cite{induced_matchings} applied to our framework.  We will make use of their canonical matchings of barcodes induced by injective or surjective module homomorphisms.  
\begin{definition}(see Section 4 in \cite{induced_matchings})
\label{canonical}
Let $P$ be totally ordered, and let $I = \bigoplus\limits_{s \in S}I_s, M = \bigoplus\limits_{t\in T} M_t$ be in $\mathcal{C}$.  Let $f$ be a module homomorphism from $I \xrightarrow{f} M$.  Then, 
\begin{enumerate}[(i)]
\item if $f$ is surjective, let $\Theta(f)$ from $B(M)$ to $B(I)$ be the canonical matching of barcodes.
\item if $f$ is injective, let $\Theta(f)$ from $B(I)$ to $B(M)$ be the canonical matching of barcodes. 

\end{enumerate}
\end{definition}
Recall from \cite{induced_matchings}, $\Theta$ is categorical on injections or surjections.  That is, $f = g \circ h$, $f, g, h$ surjections implies $\Theta (f) = \Theta (h) \circ \Theta (g)$.  And dually, $f = g \circ h$, $f, g, h $ injections implies $\Theta (f) = \Theta (g) \circ \Theta (h)$.  

The authors wish to emphasize that the above statements holds for any permissable enumeration on each barcode.  That is to say, for each module $M$, all isomorphic elements of the barcode $B(M)$ may be enumerated arbitrarily.  This enumeration is then fixed.  In one instance, it will be convenient (though not necessary) to choose explicitly an enumeration for a particular barcode.  

We now establish some additional properties of convex modules for $1$-Vees.

\begin{lemma}
\label{prematching}
Let $P=[m,n]$ be a $1$-Vee, $\Lambda$ a maximal translation on $P$. Let $\Sigma$ be the set of isomophism classes of convex modues.  Let $F,G$ be the functions $$F,G:\Sigma\rightarrow\Sigma\cup\{0\}$$ where $F(\sigma)=\sigma^{-\Lambda^2}$ and $G(\sigma)=\sigma^{+\Lambda^2}\Lambda$.  
$$\textrm{Let }\Sigma_0=\{\sigma:W(\sigma)>h(\Lambda)\} \textrm{, and }\bar{\Sigma}\textrm{ be }\Sigma_0 \cap \{\sigma \in \Sigma: \textrm{Supp}(\sigma) = [x,X] \textrm{, with } \Lambda^2 x = n\}.$$
\begin{enumerate}[(i)]
\item $F(\Sigma_0)\subset\Sigma$, and $F$ is one-to-one on $\Sigma_0$.
\item $G(\Sigma_0)\subset\Sigma$, and $G$ is one-to-one on $\Sigma_0 - \bar{\Sigma}$.  Also, $G(\bar{\Sigma}) = \{\sigma_n\Lambda\}$, where $\sigma_n$ is the convex module with support $[n]$.
\end{enumerate}
\end{lemma}
\begin{proof}

We will show that if $\sigma_1,\sigma_2\in\Sigma_0$ with$F(\sigma_1)\cong F(\sigma_2)$, then  $\sigma_1\cong\sigma_2.$  Since convex modules are characterized by their supports, say $F(\sigma_1), F(\sigma_2)$ have shared support $[x,X']$.  Then $X'$ is maximal such that $\Lambda^2X'\leq X_1$ and also such that $\Lambda^2X'\leq X_2$where $\sigma_1, \sigma_2$ have support given by $[x,X_1]$, and $[x,X_2]$ respectively.  But then by Lemma \ref{hom4}, $X_1 = X_2$, so $\sigma_1 \cong \sigma_2$.  This proves (i).

For (ii), we'll prove the contrapositive.  Suppose $\sigma_1,\sigma_2 \in \Sigma_0 - \bar{\Sigma}$ have supports given by $[x_1,X_1],[x_2,X_2]$ respectively. Suppose $\Lambda^2x_1<\Lambda^2x_2\leq n$, then by Lemma \ref{hom4}, $\Lambda x_1<\Lambda x_2$. Then, again by Lemma \ref{hom4}, $G(\sigma_1)=[\Lambda x_1,\cdot]$, $G(\sigma_2)=[\Lambda x_2,\cdot]$, which are distinct.  On the other hand, if $x_1 = x_2$, $X_1 < X_2 \leq n$, then $\sigma_1^{\Lambda^2}, \sigma_2^{\Lambda^2}$ have supports given by $[\Lambda^2x_1,X_1], [\Lambda^2x_1,X_2]$ respectively.  But then, since only $X_2$ is possibly equal to $n$, the right endpoint of the support of $G(\sigma_2)$ is strictly larger than the right endpoint of the support of $G(\sigma_1)$.  

Clearly, if $\sigma \in \bar{\Sigma}$, the support of $\sigma$ is $[x,n]$, with $\Lambda^2 x = n$.  Then, by inspection, $G(\sigma) = \sigma_n\Lambda$.  Moreover, it is clear from the proof that $G^{-1}(\sigma_n\Lambda) \subseteq \bar{\Sigma}$. 

\end{proof}

\begin{prop}
\label{totallyordered}
Let $P$ be totally ordered and let $\mathcal{C}$ be the full subcategory of $A(P)$-modules consisting of direct sums of convex modules.  Let $(a,b) \in {\mathbb{N}} \times {\mathbb{N}}$ be a weight and let $D$ denote interleaving distance (corresponding to the weight $(a,b)$) restricted to $\mathcal{C}$.  \\Let  $W(M) = \textrm{min}\{ \epsilon: \textrm{Hom}(M, M\Gamma \Lambda) = 0, \Gamma, \Lambda \in  \mathcal{T}(\mathcal{P}), h(\Gamma), h(\Lambda) \leq \epsilon \}$, and let $D_B$ be the bottleneck distance on $\mathcal{C}$ corresponding to the interleaving distance and $W$.  Then, the identity is an isometry from $$(\mathcal{C},D) \xrightarrow{Id} (\mathcal{C},D_B).$$
\end{prop}

This corresponds to the case that $P$ is a $1$-Vee in Theorem \ref{main}.  The result follows from Theorem \ref{matching}.  We will proceed in the same fashion as \cite{induced_matchings}.  Before continuing, we point out that Theorem \ref{matching} (and later Theorem \ref{main}) do not say that every interleaving is diagonal (see Examples \ref{ex 4}, \ref{ex 3}).  Instead, they simply constrain the isomorphism classes of modules which admit an interleaving.
\begin{theorem}
\label{matching}
Let $P$ be totally ordered ($P$ is a $1$-Vee) and let $I = \bigoplus\limits_{s \in S}I_s, M = \bigoplus\limits_{t\in T} M_t$ be in $\mathcal{C}$.  Let $\Lambda = {\Lambda}_{\epsilon} \in \mathcal{T}(P)$ be maximal with $h(\Lambda) = \epsilon$.  Suppose there exists a $(\Lambda, \Lambda)$-interleaving between $I$ and $M$.  Then there exists a $h(\Lambda)$ matching from $B(I)$ to $B(M)$.  
\end{theorem}

The proof of the Theorem will consist of three parts.
\begin{enumerate}
\item If $W(I_s) > h(\Lambda)$, then $I_s$ is matched.
\item If $W(M_t) > h(\Lambda)$, then $M_t$ is matched.
\item If $I_s$ and $M_t$ are matched (independent of $W$), then there is a $(\Lambda,\Lambda)$-interleaving between $I_s$ and $M_t$.
\end{enumerate}

Our matching is a slight modification of the matching in \cite{induced_matchings}.  It is given by the following composition (see Definition \ref{canonical})
$$B(I) \xrightarrow{{\Theta(\rho)}^{-1}} B(im(\phi)) \xrightarrow{\Theta(\iota)} B(M\Lambda) \xrightarrow{B(M,M\Lambda)} B(M),$$
where $\iota$ is the inclusion from $im(\phi)$ into $M\Lambda$, $\rho$ is the surjection from $I$ to $im(\phi)$, and $B(M,M\Lambda)$ is the natural inclusion of barcodes induced from $M\Lambda$ to $M$ given by $B(M,M\Lambda)(M_t\Lambda) = M_t$.

Note that in \cite{induced_matchings}, it was only necessary to take the matching as far as $B(M\Lambda)$.  There, that was justified since the assignment which we call $B(M,M\Lambda)$ was a bijection between barcodes which preserved $W$.  In the present context neither of these properties hold.  Specifically, $|B(M\Lambda)|$ may be strictly smaller than $|B(M)|$.  Moreover, either of $W(M_t), W(M_t \Lambda)$ may be strictly larger than the other.  The detailed schematic below displays all relevant convex modules.  This will be useful in the proof.
\vspace*{.5 in}

\begin{center}
\includegraphics[scale=1.5]{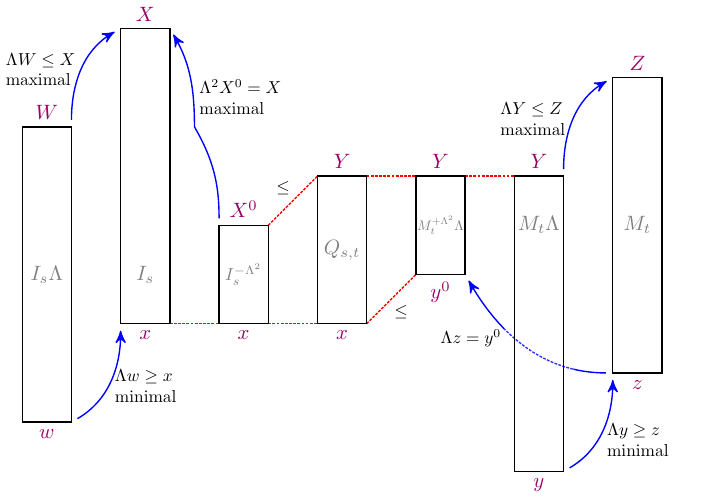}
\end{center}
\medskip
We now prove Theorem \ref{matching}.
\begin{proof}
First, say $I_s \in B(I)$ with $W(I_s) > h(\Lambda)$.  Then, by Lemma \ref{prime},  ${I_s}^{-\Lambda^2} \neq 0$.  Additionally, by Proposition \ref{trim}, and since induced matchings are categorical for surjections, we obtain the commutative triangle of barcodes below.

\begin{center}
\includegraphics[scale=1]{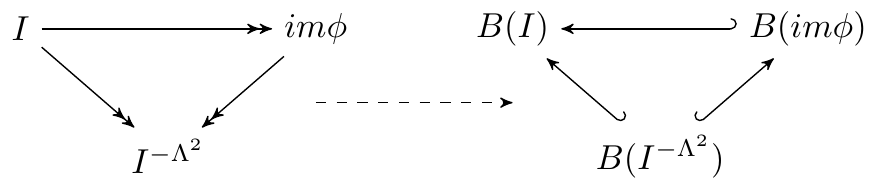}
\end{center}

But the induced matching $B(I^{{-\Lambda}^2}) \to B(I)$ sends $I_s^{-\Lambda^2} \xrightarrow{\Theta} I_s$ up to isomorphism, therefore $I_s$ is matched with an element of $B(im(\phi))$.  That is, $I_s \in im(\Theta(\rho))$.  But then, since $\Theta (\iota)$ and $ B({\Phi}_M^{{\Lambda}})$ are injections of barcodes, $I_s$ is matched with some $M_t \in B(M)$.  This establishes (1).

Next, suppose $M_t \in B(M)$ with $W(M_t) > h(\Lambda)$.  Then, by Lemma \ref{prime}, $M_t^{+\Lambda^2}\Lambda \neq 0$.  Moreover, by Proposition \ref{trim}, and since induced matchings are categorical for injections, we obtain a commutative diagram of barcodes for any choice of admissible emumeration.  It is convenient to specify a particular enumeration for $B(M^{+\Lambda^2}\Lambda)$.  This is done as follows;
\begin{itemize}
\item For $\sigma \in B(M^{+\Lambda^2}\Lambda), \sigma \ncong \sigma_n\Lambda$ (see Lemma \ref{prematching}), there is no restriction on the enumeration restricted to $\{\sigma\}$.
\item For $\sigma \in B(M^{+\Lambda^2}\Lambda), \sigma \cong \sigma_n\Lambda$, enumerate $\{\sigma\}$, by $\sigma_1 = G(\tau_1) \leq G(\tau_2)= \sigma_2$ if and only if $\tau_1 \leq \tau_2$.
\end{itemize}
With this choice of enumeration, we obtain the commutative diagram below.

\begin{center}
\includegraphics[scale=1]{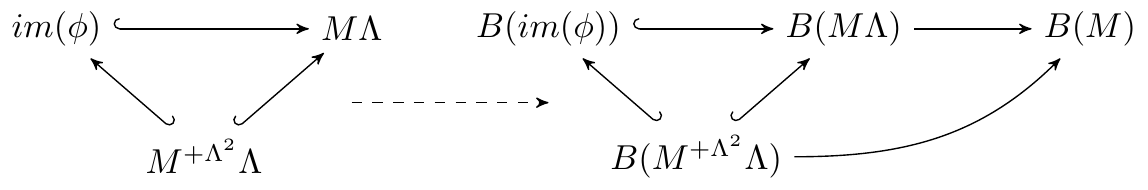}
\end{center}

Since $B(\rho)$ is an injection of barcodes, this proves (2).  

We now prove (3).  First, note that if $I_s$ and $M_t$ are matched with $h(I_s), h(M_t) \leq \epsilon$, then setting $\phi, \psi$ both equal to zero, we obtain a $(\Lambda,\Lambda)$-interleaving between $I_s$ and $M_t$.  Therefore, let
$$S' = \{s: h(I_s) > \epsilon\}\textrm{, and }  T' = \{t: h(M_t) > \epsilon\}.$$

We will write $I_s \updownarrow M_t$ when $I_s$ and $M_t$ are matched.  It remains to show that if $I_s \updownarrow M_t$, then there is a $(\Lambda,\Lambda)$-interleaving between $I_s$ and $M_t$ when,
\begin{itemize}
\item [(a)] $s \in S', t \notin T'$,
\item [(b)] $s \notin S', t \in T'$, or
\item [(c)] $s \in S', t \in T'$
\end{itemize}

Note that because of the asymmetry associated with the matching, the cases (b.) and (c.) are not identical.  Let the supports of $I_s, I_s \Lambda, M_t$ and $M_t \Lambda$ be given by $[w, W], [x,X], [y,Y]$, and $[z,Z]$ respectively.  When $s \in S'$, let $X^0$ be maximal such that $\Lambda^2 X^0 \leq X$.  That is, ${I_s^{-{\Lambda}^2}}$ has support given by $[x, X^0]$.  Similarly, when $t \in T'$, let $y^0 = \Lambda z$, so then $M_t^{+\Lambda^2}\Lambda$ has support given by $[y^0,Y]$.  Note that if $z \notin im(\Lambda)$, we have that $\Lambda z < \Lambda^2 y$ and $y = m$.

Proceeding as in \cite{induced_matchings}, by Proposition \ref{inj surj}, if $I_s \updownarrow M_t$, then we have the relations
$$y \leq x \leq Y \leq X. $$
Hence, there is a non-zero homomorphism from $I_s \to M_t\Lambda$.  Therefore, set $\Phi_{I_s,M_t\Lambda} = \chi([x,Y]) = {\phi}'$.  This will be one of our interleaving morphisms.  We next define our second interleaving morphism.  We must show that if one of (a), (b), or (c) is satisfied, we have the relations,
$$w \leq z \leq W \leq Z .$$

By inspection, it suffices to show the following statements:
\begin{enumerate}[(i)]
\item If $t\in T'$, then $w\leq z$.
\item If $s\in S'$, then $z\leq W$ and $W\leq Z$.
\item If $s\in S'$ and $t\not\in T'$, then $w\leq z$.
\item If $s\not\in S'$ and $t\in T'$, then $z\leq W$ and $W\leq Z$.
\end{enumerate}

We now prove (i) through (iv).  First, if $t \in T'$, then $\Lambda z=y_0$. Also, $w$ is minimal such that $\Lambda w\geq x$. As $x\leq y_0$, $x\leq\Lambda z$, and so $w\leq z$ by minimality.  This proves (i).  

Next, say  $s \in S'$.  Then, $z\leq\Lambda y$ by definition. Also, since $x \leq y$ we have that $\Lambda y\leq\Lambda x$. As $W$ is maximal such that $\Lambda W\leq X$, and $s\in S'$, we have $\Lambda^2x\leq X$. Therefore, $\Lambda x \leq W$.  Since $\Lambda(\Lambda x)\leq X$, $\Lambda x\leq W$. Therefore $z\leq\Lambda y\leq\Lambda x\leq W \textrm{ as required.} $  Continuing, since $s\in S'$, $X_0$ is maximal such that $\Lambda^2X_0=X$. By the maximality of $W$, $\Lambda X_0=W$. But then we have $W=\Lambda X_0\leq\Lambda Y\leq Z,$ since $im(\phi)$ includes into $I^{-\Lambda^2}$.  This proves (ii).

Now, suppose $s\in S'$, $t\not\in T'$.  If $x\geq\Lambda m$, then $\Lambda w=x$, and so $\Lambda^3w=\Lambda^2x\leq X$, since $s \in S'$. But then, $W\geq\Lambda^2w$.  Hence, since $t\not\in T'$, we have $\Lambda^2w\leq W\leq Z<\Lambda^2z$.  The result follows from monoticity.  On the other hand, if $x<\Lambda m$, then $w=m$, so $w\leq z$.  Thus we have shown (iii).

Lastly, say $s \notin S', t \in T'$.  We must establish $z\leq W$ and $W\leq Z$.  First, since $t\in T'$, we have that $\Lambda z=y_0\leq Y\leq X$.  Since $W$ is maximal with $\Lambda W\leq X$, it follows that $z\leq W$.  Next, note that if $t \in T'$, then $\Lambda W = X$, since $\Lambda W \neq X \implies X \notin im(\Lambda)$.  Then $\Lambda^2z=\Lambda(\Lambda z)\leq Z$, so $Y\geq\Lambda z$.  Therefore $X$ is in $im(\Lambda)$ so it must be the case that $\Lambda W=X$.  But then since $s\not\in S', t \in T'$, we have $\Lambda W=X<\Lambda^2x\leq\Lambda^2y_0=\Lambda^3z\leq\Lambda Z.$ Therefore, $\Lambda W\neq n$, so by monoticity $W<Z$ as required.  This proves (iv).

Thus, we have shown that if $s\in S'$ or $t\in T'$,
$$w \leq z \leq W \leq Z .$$

Therefore, set $\Phi_{M_t, I_s\Lambda} = \chi([z,W]) = {\psi}'$.  This will be our second interleaving morphism.  It now remains only to show that 
$${\psi}'\Lambda \circ {\phi}' = {\Phi}_{I_s}^{{\Lambda}^2} \textrm{, and } {\phi}'\Lambda \circ {\psi}' = {\Phi}_{M_t}^{{\Lambda}^2} .$$

Thus, we have;
$$\phi'\Lambda=\chi([x,Y])\Lambda=\chi([w,Y^*])\text{,  and  }\psi'\Lambda=\chi([z,W])\Lambda=\chi([y,W^*])\textrm{, where}$$
$Y^*$ is maximal such that $\Lambda Y^*\leq Y$, and $W^*$ is the maximal with $\Lambda W^*\leq W$.

We now proceed to establish the required commutativity conditions.  First, say $s\in S'$.  We will show that $\psi'\Lambda\circ\phi'={\Phi}_{I_s,I_s{\Lambda}^2}= \chi([x,W^*])$.  Note that, by definition, $\psi'\Lambda\circ\phi'$ is a composition of module homomorphisms, and, hence, a module homomorphism.  Therefore, by Lemma \ref{hom2}, we need only show that the linear map $\chi([x,W^*])$ is non-zero at any vertex.  To do this, we will establish that $x\leq W^*$ (that is, $\chi([x,W^*])$ is non-zero at $x$).  But, $s\in S' \implies \Lambda^2x\leq X$. As $W^*$ is maximal with $\Lambda^2W^*=X$, the inequality follows.  Now say $s\notin S'$.  Then, $\psi'\Lambda\circ\phi'=0$ as required.

We now show the commutativity of the other triangle.  First, suppose that $t\in T'$.  As above, we will show that $\phi'\Lambda\circ\psi'={\Phi}_{M_t,M_t{\Lambda}^2}= \chi([z,Y^*])$.  Again, we need only demonstrate that $z\leq Y^*$. But $t\in T' \implies \Lambda^2z\leq Z$.  SInce $Y^*$ is maximal with $\Lambda^2Y^*=Z$, the result follows.  Again, if $t \notin T'$, the result is trivial.

Therefore, if $I_s \updownarrow M_t$, then there is a $(\Lambda,\Lambda)$-interleaving between $I_s$ and $M_t$ as required.  This proves (3) and finishes the proof of the theorem.
\end{proof}

In the next section we will use Theorem \ref{matching} to prove our main result.
\section{Proof of Main Results}
\label{main section}

Before proving the main results, we establish some useful facts.  This first result will allow us to make a "half matching."
\begin{lemma}
\label{function}
Let $S, T$ be sets with $S$ finite, let $x: S \to \mathscr{P}(T)$ be a function such that for all $\phi \neq S_0 \subseteq S$, 
$$\Big|\bigcup_{s \in S_0} x(s) \Big| \geq |S_0|_.$$

Then, there exists a function $F: S \to T$ such that $F$ is an injection, and for all $s, F(s) \in x(s)$.

\end{lemma}
\begin{proof}
We prove the result by induction on $|S|$.  If $|S| = 1$, the result is trivial.  Now say $|S| > 1$ and the result holds for all sets with smaller cardinality.  First, suppose there exists a non-empty subset $S_0 \subseteq S$ such that
$$\Big|\bigcup_{s \in S_0} x(s) \Big| = |S_0|_.$$ 
Let $S_0$ be a minimal non-empty subset of $S$ where equality holds.  We will show that we can define an injection $f$ from $S_0$ to $T$ with $f(s) \in x(s)$.  Pick $s_0 \in S_0$, $t_0 \in x(s_0)$ and set $f(s_0) = t_0$.  If $S_0 = \{s_0\}$ we are done, so assume $S_0 \neq \{s_0\}$.  Then, let $\bar{x}: S_0 - \{s_0\} \to \mathscr{P}(T)$, be defined by $\bar{x}(s) = x(s) - \{t_0\}$.  Now let $S'$ be a non-empty subset of $S_0 - \{s_0\}$.  Then,
\begin{eqnarray*}
\Big|\bigcup_{s \in S'} \bar{x}(s) \Big| &=& \Big|\bigcup_{s \in S'} x(s) - \{t_0\}\Big| =\Big|\big(\bigcup_{s \in S'} x(s)\big) - \{t_0\} \Big| 
\geq  |S'| +1 - 1 = |S'|,
\end{eqnarray*}
by the minimality of $S_0$.  Thus, by induction, there exists a one-to-one  function $f : S_0 - \{s_0\} \to T$ such that $f(s) \in \bar{x}(s)$.  Clearly, $f$ can be extended to an injection on all of $S_0$ .  If $S_0 = S$, set $f = F$ and we are done.  Otherwise, define $$\bar{x} : S-S_0 \to \mathscr{P}(T) \textrm{ be defined by } \bar{x} (s) = x(s) - \{f(\sigma): \sigma \in S_0\}.$$
Now, let $\bar{s_1}, \bar{s_2}, ... \bar{s_x} \in S-S_0$.  Clearly, for all $i$, $x(\bar{s_i}) = \bar{x}(\bar{s_i}) \cup T_i$ for some set $T_i \subseteq \{f(\sigma): \sigma \in S_0\}.$  Note that $$\Big| \bigcup_{i \leq k}\bar{x}(\bar{s_i})\Big| < k \implies \Big| \bigcup_{s \in S_0} x(s) \cup x(\bar{s_1}) \cup x(\bar{s_2}) \cup ... \cup x(\bar{s_k})  \Big| < |S_0| + k,$$
a contradiction.  Thus, by induction, there is an injection $\bar{f} : S-S_0 \to T$ with $\bar{f}(s) \in \bar{x} (s)$.  By construction $F = f \cup \bar{f}$ is the desired function from all of $S$ to $T$.

On the other hand, if $S$ has the property that for all $S_0 \subseteq S$, $S_0 \neq \phi$, 
$$\Big|\bigcup_{s \in S_0} x(s) \Big| > |S_0|_,$$
pick $s_1 \in S, t_1 \in x(s_1)$ and set $f(s_1) = t_1$.
Again, let
$$\bar{x} : S-\{s_1\} \to \mathscr{P}(T) \textrm{ be defined by } \bar{x} (s) = x(s) - \{t_1\}.$$
Then, for $S_0 \subseteq S-\{s_1\}$,

$$ \displaystyle \Big| \bigcup_{s \in S_0} \bar{x}(s) \Big| = \Big| \bigcup_{s \in S_0} x(s) - \{t_1\} \Big| \geq |S_0| + 1 -1 = |S_0|.$$

Since $|S-\{s_1\}| <|S|$, the result holds by induction.
\end{proof}
\begin{ex}
$S = \{1,2,3,4,5\}, T=\{a,b,c,d,e\}$ the function $x$ given by 
$$1 \to \{a,b,d \}, 2 \to \{b,c, e\}, 3 \to \{a, c,d\},  4 \to \{d\}, 5 \to \{e\}.$$
A matching is constructed by setting $f(4) = d,$ and $ f(5) = e$.  Then, one can choose any bijection from $\bar{f}: \{1,2,3\} \to \{a,b,c\}.$  We glue $f$ and $\bar{f}$ to obtain an injection $F$ from $S$ to $T$.
\end{ex}

Next we make a simple observation about interleavings.

\begin{lemma}
\label{diagonalize}
Let $P$ be any poset, $\Lambda, \Gamma \in \mathcal{T}(P)$.  Let $A, B, C, D$ be any $A(P)$-modules with $\phi, \psi$ a $(\Lambda,\Gamma)$-interleaving between $A \oplus B$ and $C \oplus D$.  Then, if $\textrm{Hom}(A,D\Lambda) = 0 = \textrm{Hom}(C,B\Gamma)$, then $A, C$ are $(\Lambda,\Gamma)$-interleaved and $B, D$ are $(\Lambda,\Gamma)$-interleaved.
\end{lemma}
\begin{proof}
Note that we do not assume any modules are in the category $\mathcal{C}$.  For brevity, let $f_A, f_B$ denote the canonical homomorphism from $A \to A\Gamma \Lambda$ and $B \to B\Gamma \Lambda$ respectively.  Similarly, let $g_C, g_D$ denote $C \to C \Lambda \Gamma$ and $D \to D \Lambda \Gamma$ respectively.  By decomposing $\phi$, $\psi$ into their component homomorphisms, we have;
\begin{eqnarray*}
&&\begin{bmatrix} f_A & 0 \\ 0 & f_B \end{bmatrix} = \begin{bmatrix} \psi_A^C \Lambda & \psi_A^D \Lambda \\ 0 & \psi_B^D \Lambda \end{bmatrix} \begin{bmatrix}\phi_C^A & \phi_C^B \\ 0 & \phi_D^B  \end{bmatrix} =  \begin{bmatrix} \psi_A^C \Lambda \phi_C^A & \psi_A^C \Lambda \phi_C^B + \psi_A^D \Lambda \phi_D^B \\ 0 & \psi_B^D \Lambda \phi_D^B \end{bmatrix}_{,}\\[5pt]
&&\begin{bmatrix} g_C & 0 \\ 0 & g_D\end{bmatrix} = \begin{bmatrix} \phi^A_C \Gamma & \phi_C^B  \Gamma \\ 0 & \phi_D^B \Gamma \end{bmatrix} \begin{bmatrix}\psi_A^C & \psi_A^D\\ 0 & \psi_B^D  \end{bmatrix} =  \begin{bmatrix} \phi_C^A \Gamma \psi_A^C & \phi_C^A \Gamma \psi_A^D + \phi_C^B \Gamma \psi_B^D \\ 0 & \phi_D^B \Gamma \psi_B^D \end{bmatrix} _.
\end{eqnarray*}
\medskip

Thus, by inspection, if we set $\phi_C^B, \psi_A^D = 0$, the required condition will still be satisfied.
\end{proof}

We point out that this does not say that the interleaving was initially diagonal (see Example \ref{ex 3}).

\begin{corollary}
\label{diagonalize cor}
Let $P$ be a finite poset with a unique minimal element $m$.  Let $X, Y \in \mathcal{C}$, and $\Lambda$ be a translation.  Suppose $X = \bigoplus_s X_s,$ and $ Y = \bigoplus_t Y_t$ are $(\Lambda, \Lambda)$-interleaved, and $\Lambda m = m$.  Let $S_m = \{ s \in S: X_s(m) \neq 0\}, T_m = \{t \in T: Y_t(m) \neq 0\}$.  Then, 
$$\displaystyle \bigoplus_{s \in S_m} X_s, \bigoplus_{t \in T_m }Y_t \textrm{ are }(\Lambda, \Lambda)\textrm{-interleaved, and }\bigoplus_{s \notin S_m}X_s, \bigoplus_{t \notin T_m}Y_t \textrm{ are }(\Lambda, \Lambda)\textrm{-interleaved.}$$
\end{corollary}
\begin{proof}
This follows easily from Lemma \ref{diagonalize}. 
\end{proof}

\begin{prop}
\label{propfix}
Let $P$ be an $n$-Vee.  Let $I = \bigoplus\limits_{s \in S}I_s, M = \bigoplus\limits_{t\in T} M_t$ be in $\mathcal{C}$.  Suppose for all $s, t,  I_s$ and $M_t$ are supported at $m$.  Let $\Lambda \in \mathcal{T}(P)$ with $\Lambda m = m$.  Suppose there exists a $(\Lambda, \Lambda)$-interleaving between $I$ and $M$.  Then there exists a $h(\Lambda)$ matching (in the sense of Theorem \ref{matching}) from $B(I)$ to $B(M)$.  
\end{prop}
\begin{proof}
First, we show that $|B(I)| = |B(M)|$.  Since $m$ is fixed by $\Lambda$, the commutativity of the diagram below shows that $|B(I)|= rank(f) \leq dim(M(m)) = |B(M)|$.  

\begin{center}
\begin{tikzpicture}[commutative diagrams/every diagram]
	\matrix[matrix of math nodes, name=m, commutative diagrams/every cell,row sep=.7cm,column sep=.45cm] {
	 \pgftransformscale{0.2}
		(\bigoplus\limits_{s \in S}I_s )(m)& & (\bigoplus\limits_{s \in S}I_s {\Lambda}^2)(m) = (\bigoplus\limits_{s \in S}I_s )(m)\\
		& (\bigoplus\limits_{t\in T} M_t\Lambda)(m) = (\bigoplus\limits_{t\in T} M_t)(m)& \\ };
		
		\path[commutative diagrams/.cd, every arrow, every label]
			(m-1-1) edge node {} (m-1-3)
			(m-1-1) edge node[swap] {$\phi$} (m-2-2)
			(m-2-2) edge node[swap] {$\psi \Lambda = \psi$} (m-1-3);

\end{tikzpicture}
\end{center}

Thus, by symmetry $|B(I)| = |B(M)|$.  Now, let $s \in S$.  Since $I_s \to I_s{\Lambda}^2$ is nonzero, its image is in the image of ${\psi \Lambda \phi}_{| I}$.  Thus in particular, there exists a $t \in T$ with  ${\psi}^t_s \Lambda {\phi}^s_t \neq 0$.  That is, Hom$(M_t\Lambda,I_s {\Lambda}^2) \circ \textrm{Hom} (I_s, M_t \Lambda) \neq 0$.  But by Lemma \ref{hom3} then Hom$(I_s \Lambda, M_t {\Lambda}^2)$ is also not equal to zero.  So, Hom$(M_t, I_s{\Lambda}^2)$, Hom$(I_s \Lambda, M_t {\Lambda}^2)$ are both nonzero, hence their composition is nonzero since it is defined at $m$.  But then, up to a scalar, it is the composition $M_t \to M_t {\Lambda}^2$, since Hom$(M_t, M_t {\Lambda}^2) = K$ by Lemma \ref{hom2}.  Thus, there is a $(\Lambda,\Lambda)$-interleaving between $I_s$ and $M_t$.  We have shown that whenever ${\psi}^t_s \Lambda {\phi}^s_t$ is nonzero, there is a $(\Lambda,\Lambda)$-interleaving between $I_s$ and $M_t$.

 Now for $s \in S$, let $x(s) = \{t \in T: {\psi}^t_s \Lambda {\phi}^s_t \neq 0\}$.  Let $S_0 \subseteq S$.  Then, the diagram below commutes
 
 \begin{center}
\begin{tikzpicture}[commutative diagrams/every diagram]
	\matrix[matrix of math nodes, name=m, commutative diagrams/every cell,row sep=.7cm,column sep=.45cm] {
	 \pgftransformscale{0.2}
		\bigoplus\limits_{s \in S_0}I_s & & \bigoplus\limits_{s \in S_0}I_s {\Lambda}^2\\
		& \bigoplus\limits_{\substack{t \in x(s)\\\textrm{some }s \in S_0}} M_t& \\ };
		
		\path[commutative diagrams/.cd, every arrow, every label]
			(m-1-1) edge node {} (m-1-3)
			(m-1-1) edge node[swap] {} (m-2-2)
			(m-2-2) edge node[swap] {} (m-1-3);

\end{tikzpicture}
\end{center}
Hence, by evaluation at $m$, $|S_0| = rank\big\{(\bigoplus\limits_{s \in S_0}I_s)(m) \to (\bigoplus\limits_{s \in S_0}I_s{\Lambda}^2)(m)\big\} \leq |\bigcup\limits_{s \in S_0} x(s)|$.  Then, by Lemma \ref{function}, there is an injection $f$ from $S$ to $T$ with $f(s ) \in x(s)$ for all $s$.  The result follows since $|S| = |T|$ and $t \in x(s)$ implies there is a $(\Lambda, \Lambda)$-interleaving between $I_s$ and $M_t$.  
\end{proof}

We are now ready to prove our main results.

\emph{Proof of Theorem 1}
First, let $P$ be an asymmetric $n$-Vee, $P = \bigcup [n,M_i]$ with $|[m,M_{i_0}]| > |[m,M_i]|$ for $i \neq i_0$, and fix the weight $(a,b)$.  We will prove that any $(\Lambda_1,\Gamma_1)$-interleaving between $I, M \in \mathcal{C}$ produces an $\epsilon$-matching for $\epsilon = \textrm{ max }\{h(\Lambda_1),h(\Gamma_1)\}$.  Once this is established, $D_B \leq D$.  For the other inequality, note that an $\epsilon$-matching yields (after inserting appropriate zero homomorphisms) a diagonal interleaving, thus $D \leq D_B$, and hence equality.

Let $I = \bigoplus\limits_{s \in S}I_s, M = \bigoplus\limits_{t\in T} M_t$ be in $\mathcal{C}$.  If $\mathcal{V}$ is a partition of $P$, for $v \in \mathcal{V}$, let
$$S_v = \{s \in S: \textrm{ the minimal element of Supp}(I_s) \textrm{ is in }v \} .$$

Similarly, define $T_v$.  Now, suppose there is a $(\Lambda_1,\Gamma_1)$-interleaving between $I$ and $M$.  Then, by Lemma \ref{T(P)}, there exists $\Lambda = {\Lambda}_{\epsilon}$ maximal, where $\epsilon = \textrm{ max }\{h(\Lambda_1),h(\Gamma_1)\}$.  By  \cite{bubenik} since $\Lambda_1, \Gamma_1 \leq \Lambda$, there exists a $(\Lambda , \Lambda)$-interleaving between $I$ and $M$.  

First, if $\epsilon < aT+b$, then $\Lambda m = m$.  In this case, consider the partition $\mathcal{V}$ of $P$ given by 
$$\mathcal{V} =  \{(m,M_i]\} \cup \{\{m\}\}, \textrm{ and set }S_m = S_{\{m\}}, S_i = S_{(m,M_i]}.$$

Similarly, set $T_m = T_{\{m\}}$ and $T_i = T_{(m,M_i]}$.  Since $\Lambda m = m$, for all $M$ convex,
\begin{itemize}
\item Supp$(M) \subseteq (m,M_i], M\Lambda \neq 0 \implies $ Supp $M\Lambda \subseteq (m,M_i]$, and
\item $m \in $ Supp$(M) \implies m \in $ Supp$(M\Lambda)$.
\end{itemize}
Therefore, if $s \in S_m, t \in T_i$, then Hom$(I_s, M_t \Lambda) = 0 $.  Similarly, if $t \in T_m, s \in S_i$, then Hom$(M_t, I_s\Lambda) = 0 $.  Then, by Lemma \ref{diagonalize}, we may diagonalize, obtaining $(\Lambda, \Lambda)$-interleavings between 
$$\bigoplus\limits_{s \in S_m}I_s \textrm{ and }\bigoplus\limits_{t \in T_m}M_t, \textrm{ and also between } \bigoplus\limits_{s \notin S_m}I_s \textrm{ and }\bigoplus\limits_{t \notin T_m}M_t.$$

We now diagonalize further.  Again, since $\Lambda m = m$, for each $i \neq j$, $s \in S_i, t \in T_j \implies $Hom$(I_s,M_t \Lambda) = 0$ as well as the symmetric condition.  Thereofore, by applying Lemma \ref{diagonalize} repeatedly, we obtain interleavings between 
$$\bigoplus\limits_{s \in S_v} I_s \textrm{ and }\bigoplus\limits_{t \in T_v} M_t \textrm{ for all } v\in \mathcal{V}.$$   
Hence, by Proposition \ref{propfix}, we get a matching between the elements of the barcodes supported at $m$.  Also, for each $i$, $\Lambda_{|(m,M_i]}$ is a maximal translation on a totally oriented set.  Therefore, for each $i$ we acquire a matching between those elements of the barcode in $S_i$ and $T_i$ by Theorem \ref{matching}.  Thus, an $\epsilon$-matching is produced piecewise.  

Now, suppose $\epsilon = h(\Lambda) \geq aT + b$.  Then, for all convex modules $M$, $M \Lambda$ is identically $0$, or is a convex module supported in $[m, M_{i_0}]$.  Then, $M \Lambda = M \Lambda/{\mathcal{I}}_{i_0} M \Lambda$, thus any homomorphism from $N \to M\Lambda$ factors through $N/{\mathcal{I}}_{i_0}N$.  

Consider the partition 
$$\mathcal{V} = \{[m,M_{i_0}]\} \cup \{(m,M_i]:i \neq i_0\}, \textrm{ and set } S_i = S_{(m,M_i]}, S_m = S_{[m,M_{i_0}]}.$$

Similarly, define $T_m, T_i$.  Then, for $s \in S_m, t \in T_i$, Hom$(I_s, M_t \Lambda) = 0$, since $M_t \Lambda = 0.$  Since the symmetric condition holds as well, again by Lemma \ref{diagonalize} we obtain an interleaving between 
$$\bigoplus\limits_{s \in S_m}I_s \textrm{ and }\bigoplus\limits_{t \in T_m}M_t  , \textrm{ and between } \bigoplus\limits_{s \notin S_m}I_s \textrm{ and }\bigoplus\limits_{t \notin T_m}M_t  \textrm{ respectively.}$$
Since the latter interleaving corresponds to convex modules $N$ with $W(N) \leq \epsilon$, it suffices to match only convex modules with indices in $S_m$ and $T_m$.  

However, the morphisms 
$$\bigoplus\limits_{s \in S_m}I_s \xrightarrow{\phi} \bigoplus\limits_{t \in T_m}M_t\Lambda, \textrm{ and }\bigoplus\limits_{t \in T_m}M_t\xrightarrow{\psi} \bigoplus\limits_{s \in S_m}I_s\Lambda$$
factor through $\bigoplus\limits_{s \in S_m}(I_s /{\mathcal{I}}_{i_0}I_s)$  and  $\bigoplus\limits_{t \in T_m}(M_t /{\mathcal{I}}_{i_0}M_t)$ respectively.  Thus since $\Lambda_{|[m,M_{i_0}]}$ is maximal, again the result follows from Theorem \ref{matching}.

If $P$ is an $n$-Vee but is not asymmetric, then we may not use Lemmas \ref{T(P)}, and \ref{W} explicitly.  It is still the case, however, that for $\epsilon \in \{h(\Gamma):\Gamma \in \mathcal{T}(P)\}$, with $\epsilon < aT + b$, the set $\{\Lambda: h(\Lambda) = \epsilon\}$ has a unique maximal element.  Moreover, the set $\{\Lambda_{\epsilon}: \epsilon < aT + b\}$ is still totally ordered.  Thus, if $I$, $M$ are $(\Lambda, \Gamma)$-interleaved with max$\{h(\Lambda),h(\Gamma)\} < aT+b$ the proof above still goes through.  On the other hand, when $P$ is not asymmetric, for all convex modules $\sigma$, $W(\sigma) \leq  aT+b = aT_{i_0}+b$.  Therefore, though there is not a unique translation with height corresponding to this value, interleavings of this height always produce empty matchings.
\qed

\section{Examples}
\label{section ex}

We conclude with some examples. First, in Example \ref{ex 4}, we decompose an interleaving as in the proof of Theorem \ref{main}.  We also compute the varieties (see Proposition \ref{variety}) corresponding to two interleavings.  Along the way, we construct some non-diagonal interleavings.  In this section, if $M$ is convex with support given by $S$, we write $M \sim S$.




\begin{ex}
\label{ex 4}
Let $P$ be the $2$-Vee, $P=[m,x_3]\cup[m,y_6]$ and let $(a,b)$ be a weight.  Let $\Lambda = \Lambda_a$.
Consider the following convex modules.
\begin{center}
\includegraphics[scale=1]{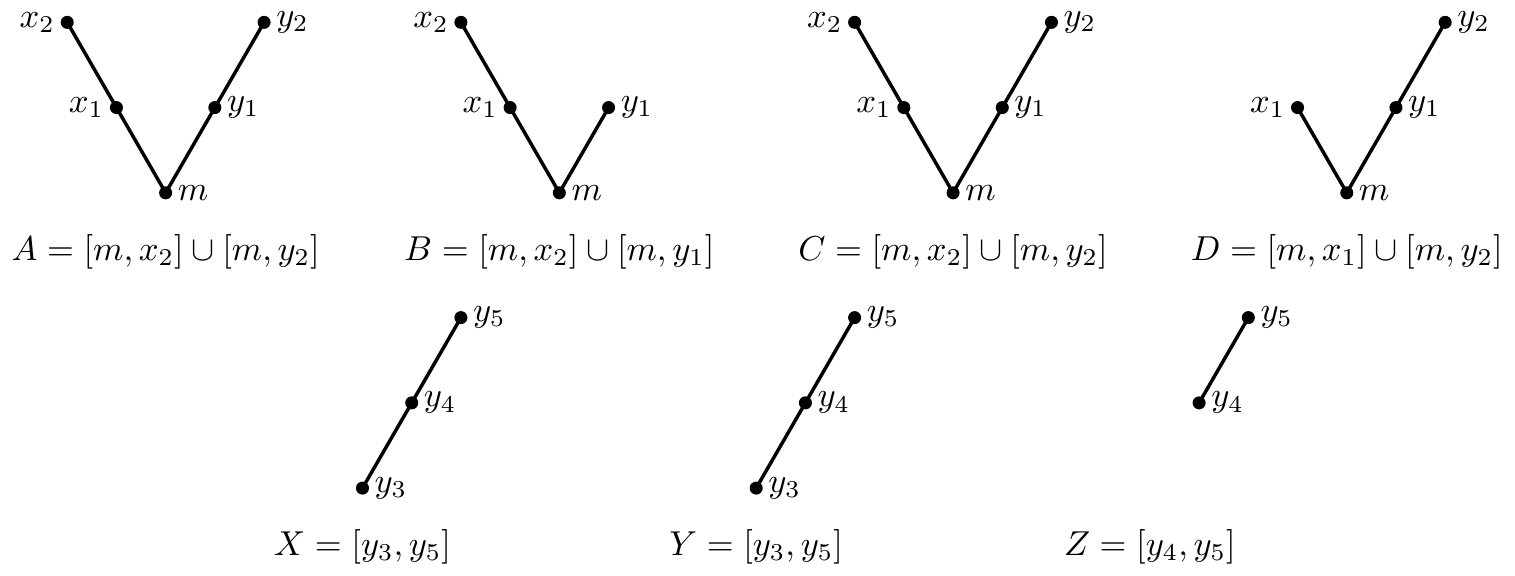}
\end{center}

Let $I=A\oplus B\oplus X$ and $M=C\oplus D\oplus Y\oplus Z$. We will decompose an arbitrary $(\Lambda,\Lambda)$-interleaving between $I$ and $M$ as in the proof of Theorem \ref{main}.  Then, we will calculate the varieties (see Remark \ref{variety}) corresponding to the "factored" interleavings the decomposition produces on the appropriate partition of the barcodes $B(I)$ and $B(M)$.  

First, let $\phi',\psi'$ be any $(\Lambda,\Lambda)$-interleaving between $I$ and $M$.  By Lemma \ref{diagonalize}, since $\Lambda m=m$, there exist
\begin{itemize}
\item a $(\Lambda,\Lambda)$-interleaving between $A\oplus B$ and $C\oplus D$, and
\item a $(\Lambda,\Lambda)$-interleaving between $X$ and $Y\oplus Z$.
\end{itemize}
We treat these separately, referring to each in turn as $\phi,\psi$.  First, we factor each $\phi, \psi$ into their corresponding summands, adopting the previous notation.  For example $\phi^X_Y : X \to Y\Lambda$.  Since $I, M \in \mathcal{C}$, we know that 
 $$\phi^X_Y=\lambda\Phi_{X,Y\Lambda}\textrm{, and }\phi^X_Y\Lambda=\lambda\Phi_{X\Lambda,Y\Lambda^2}.$$ 
 
 Of course, all other similar identities hold as well.  We first concentrate on the modules supported in $(m,M_y]$.  Thus, we have the diagrams below.
\begin{center}
\includegraphics[scale=1]{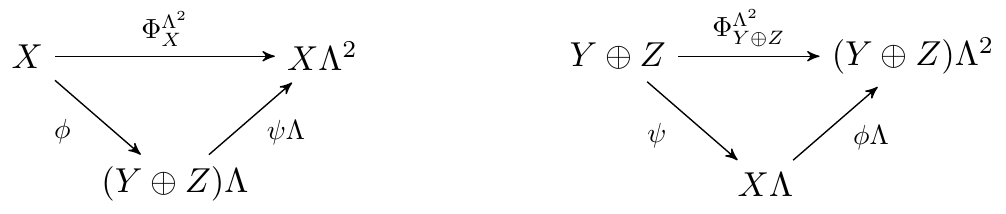}
\end{center}
Therefore, we have the matrices of module homorphisms
$$\phi=\begin{bmatrix}
\phi^X_Y=\alpha\Phi_{X,Y\Lambda} \\
\phi^X_Z=\beta\Phi_{X,Z\Lambda} \\
\end{bmatrix}
\text{ and }
\psi=\begin{bmatrix}
\psi^Y_X=\lambda\Phi_{Y,X\Lambda} & \psi^Z_X=\mu\Phi_{Z,X\Lambda} \\
\end{bmatrix}$$

Since $\phi,\psi$ is a $(\Lambda_a,\Lambda_a)$-interleaving between $X$ and $Y\oplus Z$ equations (2) and (3) below must hold. First, 
\begin{eqnarray}\Phi_X^{\Lambda^2}=
\begin{bmatrix}
\Phi^{\Lambda^2}_X \\
\end{bmatrix}
=
\begin{bmatrix}
(\psi^Y_X\Lambda)\phi^X_Y+(\psi^Z_X\Lambda)\phi^X_Z \\
\end{bmatrix}=
\begin{bmatrix}
(\lambda\alpha+\mu\beta)\Phi_X^{\Lambda^2} \\
\end{bmatrix}.\end{eqnarray}

And then,  
\begin{eqnarray}\Phi_{Y\oplus Z}^{\Lambda^2}=
\begin{bmatrix}
\Phi^{\Lambda^2}_Y & 0 \\
0 & \Phi^{\Lambda^2}_Z \\
\end{bmatrix}
=
\begin{bmatrix}
\alpha\lambda\Phi_{Y,Y\Lambda^2} & \alpha\mu\Phi_{Z,Y\Lambda^2} \\
\beta\lambda\Phi_{Y,Z\Lambda^2} & \beta\mu\Phi_{Z,Z\Lambda^2} \\
\end{bmatrix}.\end{eqnarray}
Note that the above equations define the variety $V^{\Lambda,\Lambda}(X,Y \oplus Z)$, because in the notation of Proposition \ref{variety}, no variables are deleted between $R$ and $\bar{R}$, and $Q$ and $\bar{Q}$ respectively.  For computational purposes, it is convenient to arrange the relevant convex modules side by side:
\begin{center}
\includegraphics[scale=1]{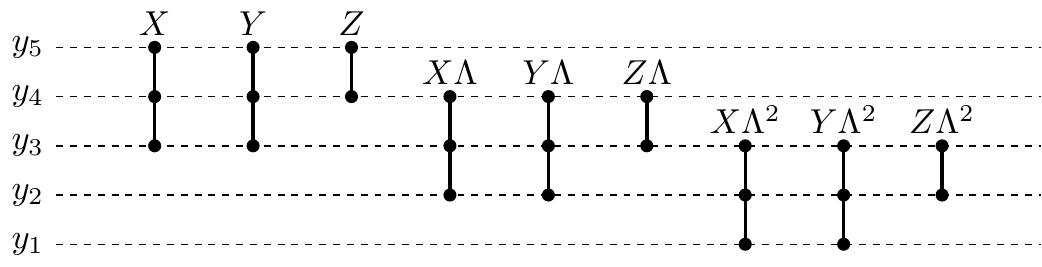}
\end{center}

Notice that evaluating (2) at all elements of $P$ we have
$$\left(\Phi_X^{\Lambda^2}\right)(y_3)=\begin{bmatrix} 1 \\ \end{bmatrix}\text{ and }\left(\Phi_X^{\Lambda^2}\right)(i)=\begin{bmatrix} 0 \\ \end{bmatrix} \textrm{ for any } i\neq y_3.$$

Similarly, evaluating (3) we have  $$\left(\Phi_{Y\oplus Z}^{\Lambda^2}\right)(y_3)=\begin{bmatrix} 1 & 0 \\ 0 & 0 \\ \end{bmatrix}\text{ and }\left(\Phi_{Y\oplus Z}^{\Lambda^2}\right)(i)=\begin{bmatrix} 0 & 0 \\ 0 & 0 \\ \end{bmatrix} \textrm{ for }i\neq y_3.$$

Therefore, we get 
the following system of equations:

$$\lambda\alpha+\mu\beta=1, \alpha\lambda=1, \beta\lambda=0.$$

Thus, $V^{\Lambda,\Lambda}(X,Y\oplus Z)$ is the affine variety with coordinate ring $K[\lambda,\mu,\alpha,\beta]$ modulo the ideal $\langle \lambda\alpha+\mu\beta-1, \alpha\lambda-1$, $\lambda\beta\rangle$ (see Proposition \ref{variety}).

This corresponds to one choice of parameter in $K^*$ and one in $K$.\\



Next consider the modules supported at $m$.  We now compute the variety $V^{\Lambda,\Lambda}(A\oplus B, C \oplus D)$.  Again, we must have the commutative triangles below.

\begin{center}
\includegraphics[scale=1]{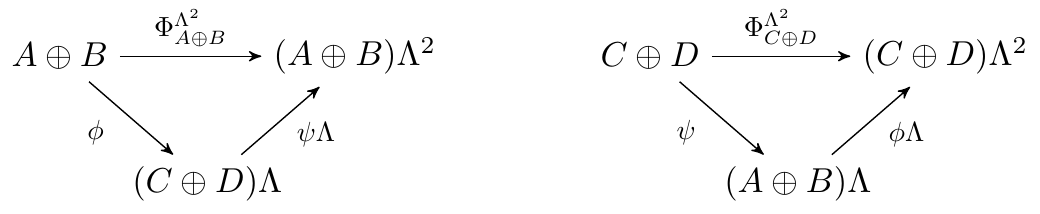}
\end{center}

Decomposing $\phi, \psi$ we have,
$$\phi=\begin{bmatrix}
\phi^A_C=e\Phi_{A,C\Lambda} & \phi^B_C=f\Phi_{B,C\Lambda} \\
\phi^A_D=g\Phi_{A,D\Lambda} & \phi^B_D=h\Phi_{B,D\Lambda} \\
\end{bmatrix}
\text{ and }
\psi=\begin{bmatrix}
\psi^C_A=i\Phi_{C,A\Lambda} & \psi^D_A=j\Phi_{D,A\Lambda} \\
\psi^C_B=k\Phi_{C,B\Lambda} & \psi^D_B=l\Phi_{D,B\Lambda} \\
\end{bmatrix}.$$

Since $\phi,\psi$ is a $(\Lambda,\Lambda)$-interleaving between $A\oplus B$ and $C\oplus D$ (and since no variables are eliminated from $Q$ to $\bar{Q}$) we have,

$$\Phi_{A\oplus B}^{\Lambda^2}=
\begin{bmatrix}
\Phi^{\Lambda^2}_A & 0 \\
0 & \Phi^{\Lambda^2}_B \\
\end{bmatrix}
=
\begin{bmatrix}
(ei+gj)\Phi_{A,A\Lambda^2} & (fi+hj)\Phi_{B,A\Lambda^2} \\
(ek+gl)\Phi_{A,B\Lambda^2} & (fk+hl)\Phi_{B,B\Lambda^2} \\
\end{bmatrix}$$
Evaluating everything at $m$, we obtain
$$\begin{bmatrix}
ei+gj & fi+hj \\
ek+gl & fk+hl \\
\end{bmatrix}
=
\begin{bmatrix}
1 & 0 \\
0 & 1 \\
\end{bmatrix}.$$
Since evaluation at any other element of the poset makes all equations trivial, this identity (along with the redundant one obtained from ${\Phi}_{C \oplus D}^{{\Lambda}^2}$) is the only necessary condition.  Therefore, the space of $(\Lambda,\Lambda)$-interleavings between $A\oplus B$ and $C\oplus D$, $V^{\Lambda,\Lambda}(A \oplus B, C \oplus D)$ is $Gl_2(K)$, the variety of invertible $2\times 2$ matrices.  In particular, many interleavings between $A \oplus B$ and $C \oplus D$ are as far from diagonal as possible.

By choosing a point in each variety separately, we obtain an interleaving between $I$ and $M$.  Note that although we produced many interleavings, we did not classify the $(\Lambda,\Lambda)$-interleavings between $I$ and $M$.  This is clear, as we passed from our original $\phi', \psi'$ to a pair of separate interleavings on a partition of each barcode.  In the next example we compute the full variety of interleavings between two elements of $\mathcal{C}$.

\end{ex}

\begin{ex}
\label{ex 3}
Let $P$ be again be the $2$-Vee, $P =[m,x_3]\cup[m,y_6]$.  Let $(a,b)$ be a weight, and set $\Lambda=\Lambda_{2a}$.
\begin{center}
\includegraphics[scale=.8]{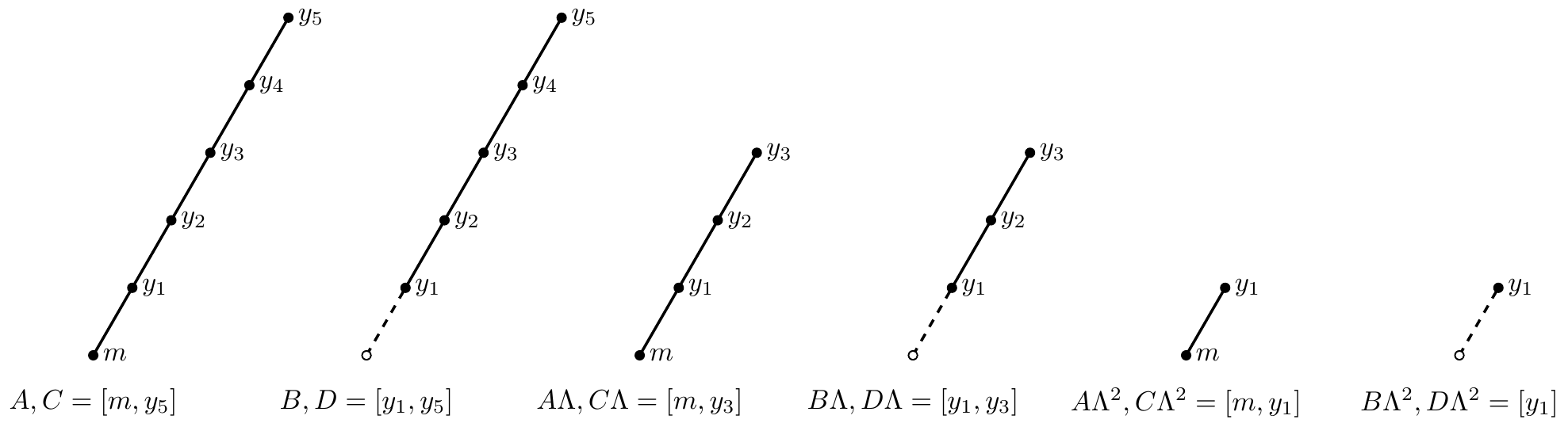}
\end{center}

Let $I=A\oplus B$ and $M=C\oplus D$, and let $\Lambda = \Lambda_{2a}$.  We will calculate$V^{\Lambda,\Lambda}(I,M)$, the variety corresponding to all $(\Lambda,\Lambda)$-interleavings between $I$ and $M$ (see Proposition \ref{variety}).  First, consider an arbitrary such interleaving, $\phi,\psi$.  Of course, as always this yields the standard commutative triangles.  Since $\Lambda m=m$, there exist no non-zero morphisms from $A$ into any module not containing the minimal $m$.  Moreover, $m \in \textrm{Supp}(\sigma \Lambda)$ if and only if $m \in \textrm{Supp}(\sigma)$.  Hence, using our standard notation, it must be the case that $\phi^A_D$ and $\psi^C_B$ are identically zeros. Then, as matrices of module homomorphisms, we have

$$\phi=\begin{bmatrix}
\phi^A_C & \phi^B_C \\
0 & \phi^B_D \\
\end{bmatrix}=
\begin{bmatrix}
\lambda\Phi^A_{C\Lambda} & \rho\Phi^B_{C\Lambda} \\
0 & \mu\Phi^B_{D\Lambda} \\
\end{bmatrix}
\text{ and }
\psi=\begin{bmatrix}
\psi^C_A & \psi^D_A \\
0 & \psi^D_B \\
\end{bmatrix}=
\begin{bmatrix}
\alpha\Phi^C_{A\Lambda} & \gamma\Phi^D_{A\Lambda} \\
0 & \beta\Phi^D_{B\Lambda} \\
\end{bmatrix}$$

Since $\phi$ and $\psi$ constitute a $(\Lambda,\Lambda)$-interleaving, we obtain equations (4) and (5) below.
\begin{eqnarray}\Phi_I^{\Lambda^2}=\Phi_{A\oplus B}^{\Lambda^2}=
\begin{bmatrix}
\Phi_A^{\Lambda^2} & 0 \\
0 & \Phi_B^{\Lambda^2} \\
\end{bmatrix}
=
\begin{bmatrix}
\alpha\lambda\Phi_{A,A\Lambda^2} & (\alpha\rho+\gamma\mu)\Phi_{B,A\Lambda^2} \\
0 & \beta\mu\Phi_{B,B\Lambda^2} \\
\end{bmatrix}
\end{eqnarray}

\begin{eqnarray}\Phi_M^{\Lambda^2}=\Phi_{C\oplus D}^{\Lambda^2}=
\begin{bmatrix}
\Phi_C^{\Lambda^2} & 0 \\
0 & \Phi_D^{\Lambda^2} \\
\end{bmatrix}
=
\begin{bmatrix}
\lambda\alpha\Phi_{C,C\Lambda^2} & (\lambda\gamma+\rho\beta)\Phi_{D,C\Lambda^2} \\
0 & \mu\beta\Phi_{D,D\Lambda^2} \\
\end{bmatrix}\end{eqnarray}
Evaluating (4), we see that 

$$\left(\Phi_I^{\Lambda^2}\right)(m)=\begin{bmatrix} 1 & 0 \\ 0 & 0 \\ \end{bmatrix}, \left(\Phi_I^{\Lambda^2}\right)(y_1)=\begin{bmatrix} 1 & 0 \\ 0 & 1 \\ \end{bmatrix}, \textrm{ and }\left(\Phi_I^{\Lambda^2}\right)(x)=\begin{bmatrix} 0 & 0 \\ 0 & 0 \\ \end{bmatrix} \textrm{ for }x \neq m, y_1$$

From the evaluation at $y_1$, we get the restrictions: 
$$\alpha\lambda=1, \beta\mu=1, \alpha\rho+\gamma\mu=0,$$
the last of which only appears because $\Phi_{B,A\Lambda^2}\neq 0$. 

When evaluating at $m$, we obtain the redundant constraint $\alpha\lambda=1$.  Since every homomorphism in $(4)$ has support contained in $[m,y_1]$, so there are no further restrictions from (4).

By inspection, evaluating (5) obtains no new conditions, since
$$\alpha \rho + \gamma \mu = 0 \iff \rho = -\lambda \gamma \mu \iff \lambda \gamma + \rho \beta = 0.$$

Therefore, $V^{\Lambda,\Lambda}(I,M)$ is the affine variety with coordinate ring $K[\lambda,\mu,\rho,\alpha,\beta,\gamma]$ modulo the ideal $\langle \lambda\alpha-1, \mu\beta-1, \alpha \rho + \gamma \mu\rangle$.

Thus, the interleavings are parametrized by two elements of $K^*$ and one element of $K$.  

\end{ex}

We point out that in both Examples \ref{ex 4} and \ref{ex 3} there were additional degrees of freedom for interleavings not seen when one passes to a matching.  In Example \ref{ex new} we realize the interleaving distance as the minimum height of a translation with non-empty variety.

\medskip
\begin{ex}
\label{ex new}

\medskip

Let $P$ be the $1$-Vee $P=[m,x_3]$, and let $(a,b)$ be a weight, with $a < b$. Consider the convex modules
\begin{center}
\includegraphics[scale=1]{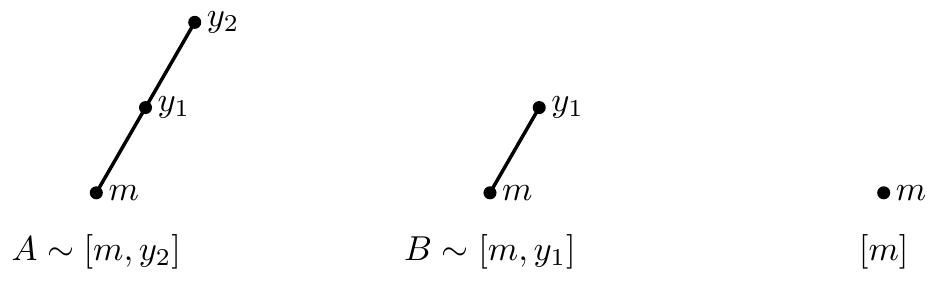}
\end{center}
We will calculate the variety $V^{\Lambda,\Lambda}(I,M)$ for various choices of translations.  We will use these calculations to point out the interleaving distance as in Remark \ref{variety distance}.  Whatever translations we consider, we always set

$$\phi=\begin{bmatrix}
\phi^A_B=\alpha\Phi_{A,B\Lambda} \\
\end{bmatrix}
\text{ and }
\psi=\begin{bmatrix}
\psi^B_A=\beta\Phi_{B,A\Lambda}. \\
\end{bmatrix}$$ 

First, let $\Lambda=\Lambda_0$, the identity translation.  We can see that there is no $(\Lambda_0,\Lambda_0)$-interleaving between $A$ and $B$, since $\Phi_A^{\Lambda^2}\neq0$ but $\textrm{Hom}(B,A\Lambda)=0$.  The variety of $(\Lambda_0,\Lambda_0)$-interleavings therefore, must be the empty variety.  We recover this via the equation,

$$[{\Phi}_A^{\Lambda^2}] = \bar{Q}\cdot R = [0]\cdot [\alpha{\Phi}_{A,B\Lambda}] $$

(using the notation of Proposition \ref{variety}) which has no solution, say when evaluated at $m$.  Thus, for $\Lambda = \Gamma = \Lambda_0$, the variety $V^{\Lambda, \Gamma}(I,M)$ is empty, as required
\vspace{.2 in}.


Now, say $\Lambda=\Lambda_a$.  Since $a < b$, $\Lambda_a$ is the next largest maximal translations.  Then, $A\Lambda \cong B$ and $B\Lambda \cong C$.  Moreover, $\textrm{Hom}(A,B\Lambda), \textrm{Hom}(B,A\Lambda) \cong K$.  The space of interleavings are defined by the equations

$$\begin{bmatrix}
\Phi_A^{\Lambda^2} \\
\end{bmatrix}
=
\bar{Q} \cdot R
=
\begin{bmatrix}
\alpha\beta(\Phi_{B\Lambda,A\Lambda^2}\cdot\Phi_{A, B\Lambda})
\end{bmatrix}, \begin{bmatrix}
\Phi_B^{\Lambda^2} \\
\end{bmatrix}
=
\bar{R}\cdot Q
=
\begin{bmatrix}
0 \cdot \beta(\Phi_{A\Lambda,B\Lambda^2}\cdot\Phi_{B,A\Lambda})
\end{bmatrix}$$ 

\medskip
Note that the variable $\alpha$ is absent from $\bar{R}$, since Hom$(A\Lambda,B\Lambda^2) = 0$.  The first equation yields $\alpha \beta =1$ by evaluating at $m$, and the second equation is consistent and trivially satisfied since ${\Phi}_B^{\Lambda^2} = 0$.  

Therefore, the space of $(\Lambda_a,\Lambda_a)$-interleavings corresponds to the affine variety with coordinate ring $K[\alpha,\beta]$ modulo the ideal $\langle\alpha\beta-1\rangle$. This corresponds to a choice of one parameter in $K^*$.  Also, since $\epsilon = a$ corresponds to the first non-zero variety, we can see that $D(A,B) = a$.

\vspace{.2 in}
Now suppose that $\Lambda=\Gamma= \Lambda_{2a}$.  Then, $A\Lambda \cong C$ and $B\Lambda \cong 0$.  Also, $\textrm{Hom}(A,B\Lambda)=0$ and $\textrm{Hom}(B,A\Lambda)\cong K$.

The space of interleavings are defined by the equations

$$\begin{bmatrix}
\Phi_A^{\Lambda^2} \\
\end{bmatrix}
=
\bar{Q} \cdot R
=
\begin{bmatrix}
0 \cdot 0 (\Phi_{B\Lambda,A\Lambda^2}\cdot\Phi_{A, B\Lambda})
\end{bmatrix}, \begin{bmatrix}
\Phi_B^{\Lambda^2} \\
\end{bmatrix}
=
\bar{R}\cdot Q
=
\begin{bmatrix}
0 \cdot \beta(\Phi_{A\Lambda,B\Lambda^2}\cdot\Phi_{B,A\Lambda})
\end{bmatrix}.$$

Since ${\Phi}_A^{\Lambda^2}, {\Phi}_B^{\Lambda^2}$ are both identically zero, any value of $\beta$ satisfies the above equations.  Therefore, in this case $V^{\Lambda, \Gamma}(A,B)$ corresponds to the affine variety with coordinate ring $K[\beta]$. Of course, this corresponds to a choice of one parameter in $K$.

On the other hand, when $\Lambda=\Gamma=\Lambda_l$ for $l\geq3a$, $A\Lambda=B\Lambda=0$.  In this case, the only interleaving between $A$ and $B$ is $\phi=\psi=0$.  That is to say, tfor such translations, $V^{\Lambda,\Gamma}(A,B)$ is the $0$ variety.

Putting it all together, we obtain the curve below with values in the category of affine varieties.  Only the jump discontinuities are labeled.  In a slight abuse of notation,  we write the coordinate rings instead of their corresponding varieties.  
\begin{center}
\includegraphics[scale=1]{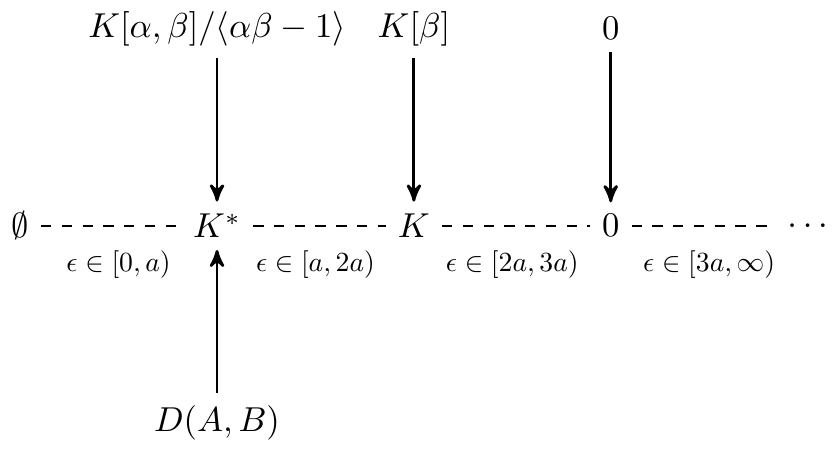}
\end{center}

\end{ex}

\tikzstyle{block} = [rectangle, draw,fill=white,text centered, rounded corners]
\tikzstyle{colored1} = [rectangle, draw, fill=blue!20,text centered, rounded corners, minimum height=4em]

\begin{rmk}
In \cite{meehan_meyer_2} we investigate the non-democratic choice of weights on a finite totally ordered set.  Here we will show that from the perspective of topological data analysis, both potential problems associated with discretizing one-dimensional persistence modules can be overcome.  Moreover, we recover the (classical) interleaving distance as a limit of discrete distances.  In future work, we will study the geometric formulation of interleaving distances (see Remark \ref{variety distance} and Example \ref{ex new}). 
\end{rmk}

\bibliography{master_bib}

\newcommand{\etalchar}[1]{$^{#1}$}
\begin{thebibliography}{ACMT05}

\bibitem[ACMT05]{strongly_simply_connected}
I.~Assem, D.~Castonguay, E.N. Marcos, and S.~Trepode.
\newblock Strongly simply connected schurian algebras and multiplicative bases.
\newblock {\em Journal of Algebra}, 283(1):161 -- 189, 2005.

\bibitem[ARS97]{auslander}
M.~Auslander, I.~Reiten, and S.~O. Smal{\o}.
\newblock {\em Representation Theory of Artin Algebras}.
\newblock Cambridge University Press, 1997.

\bibitem[Bac72]{baclawski}
Kenneth Baclawski.
\newblock Automorphisms and derivations of incidence algebras.
\newblock {\em Proceedings of the American Mathematical Society},
  36(2):351--356, 1972.

\bibitem[BdlPS11]{tame}
Thomas Brüstle, José~Antonio de~la Peña, and Andrzej Skowroński.
\newblock Tame algebras and tits quadratic forms.
\newblock {\em Advances in Mathematics}, 226(1):887 -- 951, 2011.

\bibitem[BdS13]{bubenik}
P.~{Bubenik}, V.~{de Silva}, and J.~{Scott}.
\newblock {Metrics for generalized persistence modules}.
\newblock {\em ArXiv e-prints}, December 2013.

\bibitem[Ben98a]{benson1}
D.~J. Benson.
\newblock {\em Representations and Cohomology: Volume 1}.
\newblock Cambridge University Press, 1998.

\bibitem[Ben98b]{benson2}
D.~J. Benson.
\newblock {\em Representations and Cohomology: Volume 2}.
\newblock Cambridge University Press, 1998.

\bibitem[BL13]{induced_matchings}
U.~{Bauer} and M.~{Lesnick}.
\newblock {Induced Matchings and the Algebraic Stability of Persistence
  Barcodes}.
\newblock {\em ArXiv e-prints}, November 2013.

\bibitem[BL16]{zigzag}
M.~{Bakke Botnan} and M.~{Lesnick}.
\newblock {Algebraic Stability of Zigzag Persistence Modules}.
\newblock {\em ArXiv e-prints}, April 2016.

\bibitem[Car09]{carlsson_top}
Gunnar Carlsson.
\newblock Topology and data.
\newblock {\em Bulletin of the American Mathematical Society}, 46(2):255--308,
  April 2009.

\bibitem[CCR13]{applied_1}
Joseph~Minhow Chan, Gunnar Carlsson, and Raul Rabadan.
\newblock Topology of viral evolution.
\newblock {\em Proceedings of the National Academy of Sciences},
  110(46):18566--18571, 2013.

\bibitem[CdO12]{chazal}
F.~{Chazal}, V.~{de Silva}, and S.~{Oudot}.
\newblock {Persistence stability for geometric complexes}.
\newblock {\em ArXiv e-prints}, July 2012.

\bibitem[Cib89]{cibils}
Claude Cibils.
\newblock Cohomology of incidence algebras and simplicial complexes.
\newblock {\em Journal of Pure and Applied Algebra}, 56(3):221 -- 232, 1989.

\bibitem[CIdSZ08]{carlsson_local}
Gunnar Carlsson, Tigran Ishkhanov, Vin de~Silva, and Afra Zomorodian.
\newblock On the local behavior of spaces of natural images.
\newblock {\em International Journal of Computer Vision}, 76(1):1--12, Jan
  2008.

\bibitem[CO16]{block}
J.~{Cochoy} and S.~{Oudot}.
\newblock {Decomposition of exact pfd persistence bimodules}.
\newblock {\em ArXiv e-prints}, May 2016.

\bibitem[CSEH07]{stability}
David Cohen-Steiner, Herbert Edelsbrunner, and John Harer.
\newblock Stability of persistence diagrams.
\newblock {\em Discrete {\&} Computational Geometry}, 37(1):103--120, Jan 2007.

\bibitem[CZ09]{carlsson}
Gunnar Carlsson and Afra Zomorodian.
\newblock The theory of multidimensional persistence.
\newblock {\em Discrete {\&} Computational Geometry}, 42(1):71--93, Jul 2009.

\bibitem[DN92]{thin1}
P.~Dr{\"a}xler and R.~N{\"o}renburg.
\newblock Thin start modules and representation type.
\newblock In V.~Dlab and H.~Lenzing, editors, {\em Proceedings of the Sixth
  International Conference of Representations of Algebras, Carleton University,
  Ottawa, 1992}, pages 149--163. Ottawa : Carleton University, Mathematics and
  Statistics, Ottawa, ON, 1992.

\bibitem[DN99]{thin2}
Peter Dr{\"a}xler and Rainer N{\"o}renberg.
\newblock Classification problems in the representation theory of
  finite-dimensional algebras.
\newblock In P.~Dr{\"a}xler, C.~M. Ringel, and G.~O. Michler, editors, {\em
  Computational Methods for Representations of Groups and Algebras:
  Euroconference in Essen (Germany), April 1--5, 1977}, pages 3--28.
  Birkh{\"a}user Basel, Basel, 1999.

\bibitem[DS]{drozdowski}
Grzegorz Drozdowski and Daniel Simson.
\newblock Remarks on posets of finite representation.
\newblock Institute of Mathematics. Nicholas Copernicus University.

\bibitem[EH14]{ladder}
E.~G. {Escolar} and Y.~{Hiraoka}.
\newblock {Persistence Modules on Commutative Ladders of Finite Type}.
\newblock {\em ArXiv e-prints}, April 2014.

\bibitem[ELZ02]{topological}
Edelsbrunner, Letscher, and Zomorodian.
\newblock Topological persistence and simplification.
\newblock {\em Discrete {\&} Computational Geometry}, 28(4):511--533, Nov 2002.

\bibitem[Fei76]{feinberg}
Robert~B. Feinberg.
\newblock Faithful distributive modules over incidence algebras.
\newblock {\em Pacific J. Math.}, 65(1):35--45, 1976.

\bibitem[GPCI15]{applied_3}
Chad Giusti, Eva Pastalkova, Carina Curto, and Vladimir Itskov.
\newblock Clique topology reveals intrinsic geometric structure in neural
  correlations.
\newblock {\em Proceedings of the National Academy of Sciences},
  112(44):13455--13460, 2015.

\bibitem[HNH{\etalchar{+}}16]{applied_2}
Yasuaki Hiraoka, Takenobu Nakamura, Akihiko Hirata, Emerson~G. Escolar, Kaname
  Matsue, and Yasumasa Nishiura.
\newblock Hierarchical structures of amorphous solids characterized by
  persistent homology.
\newblock {\em Proceedings of the National Academy of Sciences},
  113(26):7035--7040, 2016.

\bibitem[IK17]{mio}
M.~C. {Iovanov} and G.~D. {Koffi}.
\newblock {On Incidence Algebras and their Representations}.
\newblock {\em ArXiv e-prints}, February 2017.

\bibitem[Kle75]{kleiner}
M.~M. Kleiner.
\newblock Partially ordered sets of finite type.
\newblock {\em Journal of Soviet Mathematics}, 3(5):607--615, May 1975.

\bibitem[{Les}11]{lesnick}
M.~{Lesnick}.
\newblock {The Theory of the Interleaving Distance on Multidimensional
  Persistence Modules}.
\newblock {\em ArXiv e-prints}, June 2011.

\bibitem[Lou75]{loupias}
Mich{\`e}le Loupias.
\newblock Indecomposable representations of finite ordered sets.
\newblock In Vlastimil Dlab and Peter Gabriel, editors, {\em Representations of
  Algebras: Proceedings of the International Conference Ottawa 1974}, pages
  201--209. Springer Berlin Heidelberg, Berlin, Heidelberg, 1975.

\bibitem[MM17]{meehan_meyer_2}
K.~{Meehan} and D.~{Meyer}.
\newblock {Interleaving Distance as a Limit}.
\newblock {\em ArXiv e-prints}, October 2017.

\bibitem[Naz81]{nazarova}
L.~A. Nazarova.
\newblock Poset representations.
\newblock In Klaus~W. Roggenkamp, editor, {\em Integral Representations and
  Applications: Proceedings of a Conference held at Oberwolfach, Germany, June
  22--28, 1980}, pages 345--356. Springer Berlin Heidelberg, Berlin,
  Heidelberg, 1981.

\bibitem[Oud15]{oudot}
Steve Oudot.
\newblock {\em Persistence Theory: From Quiver Representations to Data
  Analysis}.
\newblock American Mathematical Society, 2015.

\bibitem[Rin91]{ringel}
Claus~Michael Ringel.
\newblock Recent advances in the representation theory of finite dimensional
  algebras.
\newblock In G.~O. Michler and C.~M. Ringel, editors, {\em Representation
  Theory of Finite Groups and Finite-Dimensional Algebras: Proceedings of the
  Conference at the University of Bielefeld from May 15--17, 1991, and 7 Survey
  Articles on Topics of Representation Theory}, pages 141--192. Birkh{\"a}user
  Basel, Basel, 1991.

\bibitem[SG07]{ghrist}
Vin Silva and Robert Ghrist.
\newblock Coverage in sensor networks via persistent homology.
\newblock {\em Algebraic and Geometric Topology}, 7:339--358, 04 2007.

\bibitem[Yuz81]{quadratic}
Sergey Yuzvinsky.
\newblock Linear representations of posets, their cohomology and a bilinear
  form.
\newblock {\em European Journal of Combinatorics}, 2(4):385 -- 397, 1981.

\bibitem[ZC05]{computing}
Afra Zomorodian and Gunnar Carlsson.
\newblock Computing persistent homology.
\newblock {\em Discrete {\&} Computational Geometry}, 33(2):249--274, Feb 2005.

\end{thebibliography}
\bibliographystyle{alpha}

\end{document}